\providecommand{\pgfsyspdfmark}[3]{}

\documentclass[11pt]{article}
\usepackage{mathrsfs}
\usepackage{leftindex}
\usepackage{amssymb}
\usepackage{cite}
\usepackage[all]{xy}
\usepackage{graphicx}
\usepackage{textcomp}
\usepackage{charter}
\usepackage[vflt]{floatflt}
\usepackage{mathtools}
\usepackage{enumerate}
\usepackage{wasysym}
\usepackage{rotating}
\usepackage{pdflscape}
\usepackage{fullpage}
\usepackage{authblk}
\usepackage{comment}
\usepackage{bm}
\usepackage[pdftex,colorlinks]{hyperref}
\hypersetup{
bookmarksnumbered,
pdfstartview={FitH}}

\usepackage{tikz-cd}
\usepackage{nicematrix}
\usepackage{colortbl}
\usepackage{color}
\usepackage[dvipsnames]{xcolor}

\usepackage{blkarray}
\usepackage{amsmath}
\usepackage{amsthm}
\usepackage{esint}

\usepackage{setspace}
\usepackage{amsfonts}
\usepackage{eucal}
\usepackage{latexsym}
\usepackage{mathdots}
\usepackage{mathrsfs}
\usepackage{enumitem}

\newtheorem{thm}{Theorem}[section]
\newtheorem{prop}[thm]{Proposition}
\newtheorem{lem}[thm]{Lemma}
\newtheorem{lemma}[thm]{Lemma}

\newtheorem{cor}[thm]{Corollary}
\newtheorem*{thm*}{Theorem}
\newtheorem*{cor*}{Corollary}
\newtheorem*{prop*}{Proposition}

\theoremstyle{definition}
\newtheorem{defn}[thm]{Definition}

\theoremstyle{remark}
\newtheorem{remark}[thm]{Remark}

\numberwithin{equation}{section}

\newcommand{\be}{\begin{equation}}
\newcommand{\ee}{\end{equation}}
\def\ba{\begin{eqnarray*}}
\def\ea{\end{eqnarray*}}

\newcommand{\bi}{\begin{itemize}}
\newcommand{\ei}{\end{itemize}}
\newcommand{\bn}{\begin{enumerate}}
\newcommand{\en}{\end{enumerate}}
\newcommand{\bbm}{\begin{bmatrix}}
\newcommand{\ebm}{\end{bmatrix}}
\newcommand{\bpm}{\begin{pmatrix}}
\newcommand{\epm}{\end{pmatrix}}
\newcommand{\bsm}{\left ( \begin{smallmatrix}}
\newcommand{\esm}{\end{smallmatrix} \right) }

\newcommand{\mr}{\ensuremath{\mathrm}}
\newcommand{\scr}{\ensuremath{\mathscr}}
\newcommand{\mbf}{\ensuremath{\boldsymbol}}

\newcommand{\mf}{\ensuremath{\mathfrak}}
\newcommand{\ov}{\ensuremath{\overline}}

\newcommand{\wt}{\ensuremath{\widetilde}}

\newcommand{\Ga}{\ensuremath{\Gamma}}

\newcommand{\Om}{\ensuremath{\Omega}}

\newcommand{\La}{\ensuremath{\Lambda }}
\newcommand{\la}{\ensuremath{\lambda }}
\newcommand{\om}{\ensuremath{\omega}}

\newcommand{\eps}{\ensuremath{\epsilon }}

\def\C{\mathbb{C}}
\def\R{\mathbb{R}}

\def\N{\mathbb{N}}
\def\B{\mathbb{B}}

\newcommand{\cO}{\mathcal{O}}

\newcommand{\fps}{\C\langle \langle \fz \rangle \rangle}

\def\vds{\vdots \vspace{0.2cm} }

\def\fz{\mathfrak{z}}

\def\fp{\mathbb{C} \langle \fz \rangle }
\def\fpd{\mathbb{C} \langle \fz _1, \cdots, \fz_d \rangle }
\def\fps{\mathbb{C} \langle \! \langle  \fz  \rangle \! \rangle}
\def\fpsd{\mathbb{C} \langle \! \langle  \fz _1, \cdots, \fz _d  \rangle \! \rangle}
\def\mrt{\mathrm{t}}

\def\hardy{\mathbb{H} ^2 _d}

\newcommand{\ip}[2]{\ensuremath{\left\langle {#1} , {#2} \right\rangle}}

\def\nbran{\mr{Ran} \, }
\def\nbker{\mr{Ker} \, }
\def\nbdim{\mr{dim} \, }

\def\fskew{\C \ \mathclap{\, <}{\left( \right.}   \fz  \mathclap{  \, \, \, \, \, >}{\left. \right)} \, \, }

\def\fskewd{\C \ \mathclap{\, <}{\left( \right.}   \fz _1, \cdots, \fz _d  \mathclap{  \, \, \, \, \, >}{\left. \right)} \, \, }

\def\bdn{\mathbb{B} ^{(n\times n)\cdot d}}

\def\cdm{\mathbb{C} ^{(m\times m)\cdot d}}
\def\cdn{\mathbb{C} ^{(n\times n)\cdot d}}

\def\Ups{\Upsilon}

\def\cH{\mathcal{H}}
\def\cJ{\mathcal{J}}
\def\cK{\mathcal{K}}
\def\F{\mathbb{F} _d ^+}
\def\ncu{\mathbb{C} ^{(\N \times \N) \cdot d}}
\def\ncun{\mathbb{C} ^{(\N n \times \N n) \cdot d}}

\def\rball{\mathbb{B} ^{(\N \times \N) \cdot d}}

\newcommand{\norm}[1]{{\lVert #1 \rVert}}

\newcommand{\rownorm}[1]{\lVert #1 \rVert_{\mr{row}} }

\newcommand{\colnorm}[1]{\lVert #1 \rVert_{\mr{col}} }

\newcommand\restr[2]{{
  \left.\kern-\nulldelimiterspace 
  #1 
  \right|_{#2} 
  }}

\newcommand{\opp}{\mr{opp}}
\DeclareMathOperator{\GL}{GL}

\DeclareMathOperator{\alg}{Alg}

\DeclareMathOperator{\tr}{tr}

\DeclareMathOperator{\FR}{FR}
\DeclareMathOperator{\TC}{TC}
\DeclareMathOperator{\HS}{HS}

\newcommand{\cR}{\mathscr{R}}
\renewcommand{\cO}{\scr{O}}

\title{Skew fields of non-commutative meromorphic functions affiliated to finite operator algebras}

\author[1]{M\'eric L. Augat \thanks{M. L. A was partially supported by NSF grant DMS-2155033}}
\affil[1]{\footnotesize James Madison University, \texttt{augatmx@jmu.edu}}

\author[2]{Robert T. W. Martin \thanks{R. T. W. M. partially supported by NSERC grant 2026-06481}}
\affil[2]{\footnotesize University of Manitoba, \texttt{Robert.Martin@umanitoba.ca}}

\author[3]{Eli Shamovich \thanks{E. S. was partially supported by BSF grant 2022235}}
\affil[3]{\footnotesize Ben-Gurion University of the Negev, \texttt{shamovic@bgu.ac.il}}

\author[4]{Maximilian Tornes \thanks{M.T. is partially supported by a University of Manitoba Graduate Fellowship}}
\affil[4]{\footnotesize University of Manitoba, \texttt{maximilian.tornes@umanitoba.ca}}

\date{}

\begin{document}
\maketitle
\vspace{-.75cm}

\begin{abstract}
We embed skew fields of meromorphic functions in several non-commuting variables into the topological $*-$algebra of closed linear operators affiliated to a finite von Neumann algebra, $\mf{M}$, with faithful, normal tracial state, $\tau$. Namely, let $\scr{O} _d (r)$, $r \in (0, +\infty]$, denote the rings of formal power series in $d \in \N$ non-commuting variables with ``radius of convergence" at least $r$ and let $\Xi \in \mf{M} ^{1\times d}$ be a tuple of row-norm $<r$ with the property that there exist no non-zero trace-class operators $T_1, \cdots, T_d \in \scr{B} (L^2 (\mf{M}, \tau))$ obeying $\sum _j [T_j , \Xi _j] =0$. Then the unital homomorphism $\scr{O} _d (r) \rightarrow \mf{M}$ defined by the evaluation, $\fz _j \mapsto \Xi _j$, extends to an embedding of the universal skew field of fractions of $\scr{O} _d (r)$ into the $*-$algebra of closed operators affiliated to $\mf{M}$. 

In their $L^2-$homology theory of von Neumann algebras, Connes and Shlyakhtenko defined a certain quantity, $0 \leq \Delta (\Xi) \leq d$, as the Murray--von Neumann co-dimension of the Hilbert module generated by all $d-$tuples of finite rank operators obeying the above commutator condition with $\Xi \in \mf{M} ^{1\times d}$. As demonstrated by Mai, Speicher and Yin, this $\Delta$ measures, in a certain precise sense, the algebraic independence of the components of $\Xi$. Connes--Shlyakhtenko proved that $\Delta$ is an upper bound for Voiculescu's non-microstates free entropy dimension, $\delta ^* (\Xi)$, and remark that it is of interest to determine whether these two quantities are equal. Using non-commutative function theory, we show that this is generally not the case, by constructing a triple of self-adjoint operators, $\Xi = (\Xi_1,\Xi_2,\Xi_3)$, for which this inequality is strict, $\delta^*(\Xi) < \Delta(\Xi)$. 
\end{abstract}

\section{Introduction}

Let $\mf{M}$ be a finite von Neumann algebra with faithful, normal tracial state $\tau$. 
We will typically identify $\mf{M}$ with its image under the Gelfand--Naimark--Segal (GNS) representation of $\tau$, $\mf{M} \subseteq \scr{B} (L^2 (\mf{M}, \tau))$, where $L^2 (\mf{M}, \tau) =: \cJ$ denotes the GNS space of $\tau$. 
Let $\wt{\mf{M}}$ denote the topological $*-$algebra of closed linear operators affiliated to $\mf{M}$, equipped with Nelson's measure topology \cite{Nelson}. 
If $a$ is affiliated to $\mf{M}$, we will sometimes write $a \sim \mf{M}$ in place of $a \in \wt{\mf{M}}$.

In their $L^2-$homology theory of finite von Neumann algebras, Connes and Shlyakhtenko introduced the following quantity, $\Delta$, for any $d$-tuple $\Xi := ( \Xi _1, \cdots,  \Xi _d ) \in \mf{M} ^{1 \times d}$,  
\be \Delta (\Xi) := d - \mr{dim} _{\mf{M} \ov{\otimes} \mf{M} ^{\mr{op}}} \, \left\{  T \in \FR( L^2 (\mf{M}, \tau ) ) ^{1\times d} \left| \ \sum _{j=1} ^d [T_j, \Xi _j] =0 \right. \right\} ^{-\| \cdot \|_{\mr{HS}}}, \label{Delta} \ee
where $\FR(L^2(\mf{M},\tau))$ denotes the two-sided ideal of all finite-rank linear operators on $L^2 (\mf{M}, \tau)$, the closure is taken in the Hilbert--Schmidt norm, and the dimension is the Murray-von Neumann dimension of the Hilbert module. (The above Equation (\ref{Delta}) is an equivalent and slight reformulation of the original formula for $\Delta$ of \cite{CS-hom}, following an observation of Mai--Speicher--Yin \cite[Remark 3.2]{freefield}.)
This quantity obeys $0 \leq \Delta (\Xi ) \leq d$ and it is designed to measure the co-dimension of the module of algebraic relations between the components of $\Xi$.
In the original definition of $\Delta$, the set $\Xi$ was assumed to be self-adjoint in the sense that for every $1 \leq j \leq d$, there exists $1 \leq k \leq d$, such that $\Xi_j^* = \Xi_k$.
However, as observed by Mai, Speicher and Yin in \cite{freefield}, $\Delta (\Xi)$ is still a well-defined and useful quantity for an arbitrary tuple $\Xi \in \mf{M} ^{1 \times d}$.
We will refer to $\Delta$ as the \emph{algebraic relation co-dimension}; this nomenclature is motivated by an interesting and important result of Mai, Speicher, and Yin, \emph{viz.}, that $\Delta$ measures the algebraic independence of the operators $\Xi_1,\cdots,\Xi_d$ in the following sense \cite{freefield}.
Letting $\C \langle \fz \rangle = \C \langle \fz_1,\cdots,\fz_d \rangle$ denote the free algebra on $d$ generators, we obtain a natural homomorphism of evaluation at $\Xi$, $\C \langle \fz \rangle \to \mf{M}$, $\fz _j \mapsto \Xi _j$.
In \cite[Theorem 1.1]{freefield}, Mai, Speicher and Yin prove that $\Delta (\Xi ) = d$ is maximal if and only if for every ``full" matrix $A \in \fp ^{n \times n}$, $A(\Xi)$ has no kernel or co-kernel.
Here, a matrix, $A \in \scr{R} ^{n \times n}$ over a ring, $\scr{R}$, is called \emph{full}, if there exists no factorization $A = B C$, where $B$ is an $n \times r$ matrix over $\scr{R}$, $C$ is an $r \times n$ matrix over $\scr{R}$, and $r < n$. Intuitively speaking, a matrix is non-full if its image under any homomorphism into a skew field is non-invertible.
Precisely, \cite[Theorem 1.1, (i) $\Leftrightarrow$ (v)]{freefield} states that if $\Xi \in \mf{M} ^{1 \times d}$ is any tuple, then $\Xi$ has maximal algebraic relation co-dimension, $\Delta (\Xi) =d$, if and only if the homomorphism, $\fz _j \mapsto \Xi _j$, extends to an embedding of $\fskew = \fskewd$, the free skew field of all non-commutative (NC) rational functions in the variables $\fz = (\fz_1, \cdots, \fz _d)$, into $\wt{\mf{M}}$, the $*-$algebra of closed linear operators affiliated to $\mf{M}$.
Here, the free skew field of NC rational functions is the universal localization of the free algebra, \emph{i.e.}, as first proven by S.A. Amitsur, $\fskew$ is the \emph{universal skew field of fractions of the free algebra} \cite{Amitsur}.
A general theory of localization of non-commutative rings was developed by P.M. Cohn in \cite{Cohn}, who proved in particular, that any \emph{semi-free ideal ring} (\emph{semifir}), \emph{i.e.} any ring with the property that every finitely generated left/right ideal is free as a left/right module, has a universal skew field of fractions.

In \cite{CharShly-analytic}, Charlesworth and Shlyakhtenko introduced an analytic variant of the above quantity,
$$ \Delta_1 (\Xi) := d - \mr{dim} _{\mf{M} \ov{\otimes} \mf{M} ^{\mr{op}}} \, \left\{  T \in \TC( L^2 (\mf{M}, \tau ) ) ^{1\times d} \left| \ \sum _{j=1} ^d [T_j, \Xi _j ] =0 \right. \right\} ^{-\| \cdot \|_{\mr{HS}}}, $$ where $\mr{TC} (L ^2 (\mf{M} , \tau))$ denotes the two-sided ideal of trace-class linear operators on the tracial GNS space. 
We will refer to $\Delta _1 (\Xi)$ as the analytic relation co-dimension of $\Xi$. 
Indeed, Charlesworth and Shlyakhtenko prove that if $\Delta _1 (\Xi) =d$ is maximal then $\Xi$ obeys no analytic relations in the following sense \cite[Theorem 4]{CharShly-analytic}: If $f$ is an analytic formal power series with ``radius of convergence" $r >0$ and $\max_{1 \leq j \leq d} \| \Xi_j \| < r$, then $f(\Xi) \neq 0$.
Our first main result extends both \cite[Theorem 1.1 (v) $\Rightarrow$ (iii),(i)]{freefield} and \cite[Theorem 4]{CharShly-analytic} (albeit for a different operator space norm on the $d$-tuples $\Xi$ for the latter).
To state it, we require some definitions from \cite{AMS-ncentire}.
Recall that the NC open unit row-ball is defined as
$$  \rball := \bigsqcup_{n=1}^{\infty} \bdn; \quad \quad \bdn := \left\{ X \in \C ^{n\times n} \otimes \C ^{1\times d} \left| \ \sum_{j=1}^d X_j X_j^* < 1 \right. \right\}. $$
For any $0 < r \leq + \infty$, let $\scr{O}_d(r)$ denote the ring of all NC functions that are analytic in $r  \cdot \rball$ and bounded on sub-balls (see the next section for more details).
We will typically abbreviate $\scr{O}_d := \scr{O}_d(\infty)$. In \cite{AMS-ncentire}, the authors proved that each ring $\scr{O} _d (r)$, $r \in (0, +\infty]$, is a semifir, and hence, by the universal localization theory of P. M. Cohn, each $\scr{O} _d (r)$ admits a universal skew field of fractions that we will denote by $\scr{M}_d(r)$ (and again we abbreviate $\scr{M}_d := \scr{M}_d(\infty)$).

\begin{thm*}[Theorem \ref{main1}]
    Let $\Xi = ( \Xi_1,\cdots,\Xi_d ) \in \mf{M} ^{1 \times d}$ be such that $\Delta_1(\Xi) = d$. Then, given any full matrix $A \in \scr{O} _d (r) ^{n \times n}$, $A(\Xi)$ is invertible as an affiliated operator to $\mf{M} ^{n \times n}$, \emph{viz.}, $A (\Xi)$ has no kernel or co-kernel.  Therefore, the homomorphism of evaluation at $\Xi$ extends to an embedding $\scr{M}_d (r) \hookrightarrow \wt{\mf{M}}$.
\end{thm*}

In particular, if $\Delta _1 (\Xi ) = d$ is maximal, the $d-$tuple $\Xi = (\Xi _1, \cdots, \Xi _d)$ is free of analytic relations in the sense that for any $f \in \scr{O} _d (r)$, $r > \| \Xi \| _{\mr{row}}$, $f (\Xi ) \in \mf{M}$ has no kernel or co-kernel. This can be viewed as a refinement of \cite[Theorem 4]{CharShly-analytic}, although their definition of radius of convergence differs from ours in that it corresponds to equipping $\C ^d$ with a different operator space structure (the minimum, $\ell ^\infty$, or ``polydisk" operator space structure) than the row operator space structure we employ here. 

Given any $d-$tuple, $\Xi \in \mf{M} ^{1 \times d}$, it is clear that $\Delta (\Xi ) \geq \Delta _1 (\Xi)$, by definition. Connes and Shlyakhtenko further prove that if $\Xi$ is a self-adjoint tuple then $\Delta (\Xi)$ is an upper bound for Voiculescu's non-microstates free entropy dimension, $\delta ^* (\Xi)$ \cite{Voic-FisherV,Voic-free_entropy_survey}, and Charlesworth and Shlyakhtenko show that this same inequality also holds for $\Delta _1 (\Xi)$ \cite{CharShly-analytic}. Namely, if $\Xi = (\Xi _1, \cdots, \Xi _d)$ is a self-adjoint tuple, we have the chain of inequalities,
$$ \delta^*(\Xi) \leq \Delta_1(\Xi) \leq \Delta(\Xi). $$

Our second main result is that the rightmost inequality can be strict.

\begin{thm*}[Theorem \ref{thm:counter_example}]
Let $f (x,y) = \sum_{n=0}^{\infty} \frac{1}{n!} x^n y^n \in \scr{O} _2$, and let $\Xi =(\Xi_1, \Xi_2)$ be two standard free semicirculars, where $\mf{M} = W^*(\Xi)$ is the von Neumann algebra generated by $\Xi$. Then, 
$$ \Delta_1(\Xi_1,\Xi_2, f(\Xi_1,\Xi_2)) < 3 = \Delta(\Xi_1,\Xi_2,f(\Xi_1,\Xi_2)). $$
Similarly, if $g(x,y) = \sum_{n=0}^{\infty} \frac{1}{n!} x^n y^n x^n \in \scr{O} _2$, then $$2=\delta ^* (\Xi _1, \Xi _2, g (\Xi _1, \Xi _2)) \leq \Delta _1 (\Xi_1,\Xi_2,g(\Xi_1,\Xi_2)) < 3 = \Delta(\Xi_1, \Xi_2, g(\Xi_1,\Xi_2)).$$
\end{thm*}
Note that since free semi-circulars are self-adjoint, so is $g (\Xi _1, \Xi _2)$, and $\delta ^* \leq \Delta _1$ for any self-adjoint tuple by \cite[Lemma 1]{CharShly-analytic}.
To prove the theorem, we show that the triple, $(\Xi _1, \Xi _2, g(\Xi _1, \Xi _2))$ has maximal $\Delta$ (and hence obeys no algebraic relations), but it does satisfy the analytic relation, $g (\fz_1, \fz _2) - \fz _3 =0$, and hence cannot have maximal $\Delta _1$. To the best of our knowledge these are the first such examples in the literature. The question of whether $\delta^*(\Xi) = \Delta_1(\Xi)$ for a $d$-tuple $\Xi$ of self-adjoints remains open. 

\subsection{Reader's guide} Section \ref{sec:background} contains the necessary preliminaries on free analysis and realizations of NC functions.
In Section \ref{sec:odrreal}, we provide a construction of an operator realization for any function in $\scr{O}_d(r)$ for $0 < r < \infty$ with joint spectral radius of the pencil bounded by $\frac{1}{r}$.
This result extends the construction of realizations for NC entire functions in \cite{AMS-opreal}.
Section \ref{sec:pencils} contains a key lemma that relates kernels of NC functions and kernels of linear pencils with operator coefficients using Schur complements.
Section \ref{sec:sing_lin_pencils} deals with singularities of linear operator pencils, and the ``support" projections of vectors in tensor product spaces.
We also connect trace-class cycles appearing in the case of non-maximal $\Delta_1$ to the existence of singular linear operator pencils in Section \ref{sec:sing_lin_pencils} (compare with \cite{freefield}).
Section \ref{sec:application} contains the proof of the first main result, its corollaries, and related results.
Lastly, Section \ref{sec:inequality} contains the example of operators for which there is a strict inequality $\Delta_1 < \Delta$ and the technical results necessary for the proof.

\section{Background} \label{sec:background}

\subsection{Non-commutative functions}

The $d-$dimensional, $d \in \N$, complex non-commutative (NC) universe is the collection of all $d-$tuples of square, complex matrices of all finite sizes,
$$ \ncu = \bigsqcup_{n=1}^{\infty} \cdn; \quad \quad \cdn := \C ^{n \times n} \otimes \C ^{1\times d}. $$ Here, $\C^{n \times n}$, denotes the algebra of $n \times n$ matrices over $\C$. One can view $\ncu$ as the representation space of the free algebra, $\fpd$. As such, there are two natural operations on $\ncu$, the direct sum which allows us to connect the matrix strata, and joint similarity. Namely, if $X \in \cdn$ and $Y \in \cdm$, we set
$$ X \oplus Y = \left( \bpm X_1 & 0 \\ 0 & Y_1 \end{pmatrix}, \cdots, \begin{pmatrix} X_d & 0 \\ 0 & Y_d \epm \right) \in \C ^{((n+m) \times (n+m) ) \cdot d}. $$
On each level or stratum the group $\GL_n(\C)$ acts via conjugation by simultaneous similarity.
To be more precise, if $X \in \cdn$ and $S \in \GL_n(\C)$, we write $S X S^{-1} = (S X_1 S^{-1},\cdots,S X_d S^{-1})$.

A \emph{non-commutative set}, $\Om$, is any subset of the NC universe that is closed under direct sums. Defining $\Om _n = \Om \cap \cdn$, we have that $\Om = \bigsqcup \Om _n \subseteq \ncu$. A function $f : \Om \subseteq \ncu \rightarrow \C ^{\N \times \N}$, where $\Om$ is an NC set, is called an NC function, if it satisfies the following properties:
\begin{itemize}
    \item[(i)] $f$ is \emph{graded}, \emph{i.e.} $f( \Om _n) \subseteq \C^{n \times n}$,

    \item[(ii)] $f$ \emph{respects direct sums}, \emph{i.e.} $f(X \oplus Y) = f(X) \oplus f(Y)$, and,

    \item[(iii)] $f$ \emph{respects similarities}. Namely, for every $X \in \Om _n$ and $S \in \GL_n(\C)$, such that $S X S^{-1} \in \Om _n$, $f(S X S^{-1}) = S f(X) S^{-1}$.
\end{itemize}

Since we wish to consider analytic NC functions, we need to equip the NC universe with a topology. The most natural topology for our purposes is the \emph{uniform topology}, which is defined as the topology generated by the sub-base of NC row-balls centred at any matrix-point $Y \in \cdn$,
$$ r \cdot \B ^{(\N \times \N) \cdot d} (Y) := \bigsqcup _{m=1} ^\infty \left\{ \left. X \in \C ^{(mn \times mn ) \cdot d} \right| \  \| X - I_m \otimes Y \| _{\mr{row}} < r \right\}, $$ 
where for any $Z = (Z_1, \cdots, Z_d ) \in \cdn$, $\| Z \| _{\mr{row}} := \| Z \| _{\scr{B} (\C ^n \otimes \C ^d, \C ^n ) }$ \cite[Section 7.2]{KVV}. The row-balls of various radii centred at the origin of the NC universe, $0 = (0, \cdots, 0 ) \in \C ^{1 \times d}$, $r \cdot \rball := r \cdot \rball (0)$, as described in the introduction, will be the only non-commutative domains that we will consider in this paper. 
The study of NC functions was initiated independently by Taylor \cite{Taylor2, Taylor} for the purpose of contructing a non-commutative holomorphic functional calculus for operator tuples, and by Voiculescu \cite{Voic,Voic2} for applications in his free probability theory. A related notion of NC function appeared in the ``NC Gelfand theory" of Takesaki for $C^*-$algebras \cite{Takesaki}. An elementary, but surprising fact is that local boundedness in the (say) uniform topology implies that an NC function is both holomorphic and analytic. More precisely, any NC function defined on a level-wise open (in matrix-norm) NC domain is holomorphic in the co-ordinates if and only if it is locally bounded in matrix-norm \cite{AgMcC-global,KVV}. Moreover, any such NC function, $f$, defined on a uniformly open NC domain that is locally bounded in the uniform topology has a convergent Taylor-type power series expansion about any point in its NC domain with positive radius of convergence \cite[Theorem 8.11]{KVV}. For this reason we will refer to such functions as uniformly analytic NC functions.  

In his study of the free semigroup algebra and NC Hardy space theory, Popescu introduced a radius of convergence for formal power series. We write $\fps = \fpsd$ for the ring of formal power in the $d$ non-commuting variables $\fz = (\fz _1, \cdots, \fz_d)$. An element, $f \in \fps$ can be expressed as,
\be f = \sum_{\om \in \F} \hat{f}_{\om} \, \fz ^{\om}; \quad \quad \hat{f} _\om \in \C, \label{fpseq} \ee where $\om =w_1 \cdots w_n$ is any word consisting of letters, $w_j$, chosen from the alphabet, $\{1,\cdots,d\}$, $\F$ denotes the \emph{free monoid}, the collection of all such words, and  $\fz ^\om$ denotes the \emph{free monoimial},  $\fz_{w_1} \fz_{w_2} \cdots \fz_{w_n}$. If $\om$ is as in the preceding sentence, we also write $|\om | =n$ for the \emph{length} of $\om$, \emph{i.e.} the number of letters it contains. Given a free FPS, $f$, as above in Equation (\ref{fpseq}), Popescu defines its \emph{radius of convergence} as,
$$ \frac{1}{R_f} = \limsup_{n\to \infty} \sqrt[2n]{\sum_{|\om| = n} | \hat{f} _{\om}|^2}. $$
(We apply the usual convention that $1/0 = \infty$ and $1/\infty = 0$.) Clearly, every free FPS with $R_f >0$ defines a uniformly locally bounded, hence uniformly analytic NC function in $R_f \cdot \rball$ that is, moreover, uniformly bounded on every sub-ball, $r \cdot \rball$ for $r < R_f$.
\begin{defn}
    Given any $0 < r \leq +\infty$, let 
    \[
    \scr{O}_d(r) = \left\{ f \in \fpsd \mid R_f \geq r\right\}.
    \]
\end{defn}
We will write $\scr{O}_d = \scr{O}_d(\infty)$ and call the elements of this ring NC entire functions. These rings were studied by the first three authors in \cite{AMS-ncentire}. In particular, it is proved that for $0 < r \leq +\infty$, each ring $\cO_d(r)$ is a semi-free ideal ring \cite[Theorem 4.3]{AMS-ncentire}, or \emph{semifir}, for short, in the sense of P. M. Cohn \cite{Cohn}. Here, a semifir is defined as a ring such that every finitely generated right (equivalently, left) ideal is free as a right (\emph{resp.} left) module. The algebras $\cO_d(r)$ are topological algebras with the natural topology of uniform convergence on sub-balls. In this case, not only are the finitely--generated right ideals free, but they are also automatically closed \cite[Theorem 4.5]{AMS-ncentire}. One of the main results of Cohn's non-commutative localization theory is that every semifir admits a universal skew field of fractions. To define these universal localizations, we will need the following notion of inner rank that will be used throughout the paper.

\begin{defn}
    Let $\scr{R}$ be a ring and $A \in \scr{R}^{n \times n}$ a matrix over $\scr{R}$. 
    The \emph{inner rank}, $\rho (A)$, of $A$, is the smallest integer, $\rho \geq 0$, such that there exists a factorization $A = BC$, where $B \in \scr{R}^{n \times \rho}$ and $C \in \scr{R}^{\rho \times n}$. 
    A matrix $A \in \scr{R}^{n\times n}$ is called \emph{full}, if its inner rank is $\rho(A) =n$.
\end{defn}

The universal skew field of fractions of a semifir $\scr{R}$ is a skew field, $\scr{U}$, along with an embedding $\scr{R} \hookrightarrow \scr{U}$, such that the image of $\scr{R}$ generates $\scr{U}$ as a skew field, and every full matrix over $\scr{R}$ is invertible over $\scr{U}$. We will denote by $\scr{M}_d(r)$ the universal skew field of fractions of $\cO_d(r)$. (We will again simplify $\scr{M}_d = \scr{M}_d(\infty)$.) By \cite[Remark 5.10]{AMS-ncentire}, these fields consist of rational expressions in elements of $\cO_d(r)$ under the evaluation equivalence relation analogous to \cite{KVV-diff, KVV-local}. By \cite[Lemma 5.8]{AMS-ncentire}, a matrix $A$ over $\cO_d(r)$ is full if and only if there exists $X \in r \cdot \rball$, such that $A(X)$ is invertible. Moreover, if $A$ is full then for $n \gg 0$, $\det A(X) = 0$ is a proper analytic subvariety of $r  \cdot \bdn$.

\subsection{Realizations of NC functions} \label{ss:realize}

We refer the reader to \cite{HMS-realize,AMS-opreal} for an introduction to realization theory, and to \cite{PV1,PV2,KMT-matrixreal} for an introduction to realizations about a matrix-centre.

A descriptor \emph{realization} is a triple $(A,b,c) \in \scr{B} (\cH) ^d \times \cH \times \cH$, where $A = \bsm A_1 \\ \vds \\ A_d \esm$ is a (column) $d-$tuple of bounded linear operators, $A_j$, on a complex, separable or finite--dimensional Hilbert space, $\cH$. We view $A$ as a column $d-$tuple and define the \emph{column norm} of $A$ as
\[
    \colnorm{A} := \left\| \bpm A_1 \\ \vdots \\ A _d \epm \right\|_{\scr{B} (\cH, \cH \otimes \C ^d)}.
\]
The \emph{linear pencil} of $A$ is a monic, affine-linear function on the NC universe defined by 
$$ L_A (X) := I_\cH \otimes I_m - \underbrace{\sum _{j=1} ^d A_j \otimes X_j }_{=:A \otimes X};\quad \quad X \in \cdm. $$ 
Let $\scr{D} (A)$ be the \emph{invertibility domain} of $A$, 
$$ \scr{D} (A) := \{ X \in \ncu | \ L_A (X) ^{-1} \ \exists \}. $$ 
It is readily checked that $\scr{D} (A)$ is a uniformly open NC set and that the realization, $(A,b,c)$, defines a uniformly analytic NC function on $\scr{D} (A)$, the \emph{NC transfer function} of the realization, via the formula,
$$ f (X) := (b^*  \otimes I_n)  L_A (X) ^{-1} (c \otimes I_n); \quad \quad X \in \scr{D} _n (A). $$
We will write $f \sim (A,b,c)$ to denote that $f$ is the NC transfer function of the realization $(A,b,c)$. It is clear that 
$$ \| A \otimes X \|_{\scr{B} (\C ^m \otimes \cH)} 
    = \left\| (I_\cH \otimes X_1, \cdots, I_\cH \otimes X_d) \bpm A_1 \otimes I_m \\ \vdots \\ A_d \otimes I_m \epm \right\|
    \leq \rownorm{X} \cdot \colnorm{A}, $$ 
so that $\| A \otimes X \| <1$ for any $X \in \colnorm{ A } ^{-1} \cdot \rball$. 
This shows that $\scr{D} (A) \supseteq r \cdot \rball$ where $r=\| A \| _{\mr{col}} ^{-1}$ so that any NC transfer function of a realization, $f \sim (A,b,c)$ is uniformly analytic in a uniformly open neighbourhood of $0 \in \ncu$. Conversely, by \cite[Lemma 3.2]{AMS-opreal}, any NC function, $f$, that is uniformly analytic in a uniformly open neighbourhood of $0$ is the NC transfer function of a realization, $(A,b,c)$. 
If $f \sim (A,b,c)$ and $X \in \rownorm{ A }^{-1} \cdot \bdn$ so that $X \in \scr{D} _n (A)$ and $\| X \otimes A \| <1$, $L_A (X) ^{-1}$ can be expanded as an operator-norm convergent geometric series to yield
\ba f(X) & = & b^* \otimes I_n L_A (X) ^{-1} c \otimes I_n \\
& = &  b^* \otimes I_n \sum _{\ell =0} ^\infty (A \otimes X) ^n c \otimes I_n \\ 
& = & \sum _{\om \in \F} b^* A^\om c \, X^\om. \ea 
That is, $f$ can be viewed as a free FPS with coefficients given by the \emph{moments
of the realization}, $\hat{f} _\om = b^* A^\om c$. In this sense, the realization $(A,b,c)$, of $f$, gives a power series expansion of $f$ that is ``centred at $0$". If $f \sim (A,b,c)$ and $Y \in \scr{D} _m (A)$, then the realization of $f$ can be ``translated" to the point $Y$, to obtain a representation of $f$ as a certain power series consisting of multi-linear maps that is ``centred at $Y$". (This will give the \emph{Taylor--Taylor series} expansion of $f$ about $Y$. Since we will not apply the theory of Taylor--Taylor series, we will not discuss any further details \cite{KVV,KMT-matrixreal}.) Namely, since $Y \in \scr{D} _m (A)$, this means that the linear pencil, $L_A (Y)$, is invertible. Hence, for any $X \in \C ^{(mn \times mn) \cdot d}$,
\ba L_A (X) &= & I_\cH \otimes I_{mn}   - A \otimes X \\
& = &  I_\cH \otimes I_{mn} - A \otimes Y\otimes I_n - A \otimes (X -Y\otimes I_n)  \\
& = & L_A (Y) \otimes I_n \left( I_\cH \otimes I_{mn} - L_A (Y) ^{-1}\otimes I_n \, A \otimes (X- Y \otimes I_n) \right). \ea 
This yields the representation of $f$,
\ba f(X) & = & b^* \otimes I_{mn} \left( I_\cH \otimes I_{mn} - A \otimes X \right) ^{-1} c \otimes I_{mn} \\
& = & b^* \otimes I_{mn}  \left( I_\cH \otimes I_{mn} - L_A (Y) ^{-1} \otimes I_n \, A\otimes  (X- Y \otimes I_n)  \right) ^{-1} L_A (Y) ^{-1} \otimes I_n \, c \otimes  I_{mn}, \ea for any $X \in \scr{D} _{mn} (A)$. If we define the linear maps, $\hat{A} _j : \C ^{m\times m} \rightarrow \scr{B} (\cH \otimes \C ^m)$ and $\hat{b},\hat{c} \in \scr{B} (\C ^m, \cH \otimes \C ^m)$ by 
$$ \hat{A} _j (G) := L_A (Y) ^{-1} A_j \otimes G, \quad \hat{b} (v) := b \otimes v, \quad \mbox{and} \quad \hat{c}(v) := L_A (Y) ^{-1} c \otimes v; \quad \quad G \in \C ^{m\times m}, \ v \in \C ^m, $$
then the above formula for $f (X)$ can be written as 
$$ f(X) = \hat{b} ^* \otimes I_n \, L_{\hat{A}} (X- Y \otimes I_n) ^{-1} \, \hat{c} \otimes I_n, $$ 
where the linear pencil of the row $d-$tuple, $\hat{A} := (\hat{A} _1, \cdots, \hat{A} _d)$ is defined by: 
$$ L_{\hat{A}} (X) := I_{\cH} \otimes I_{mn} - \sum _{j=1} ^d \hat{A} _j \otimes \mr{id} _n \, (X_j - Y_j \otimes I_n), $$ and where $\mr{id} _n : \C ^{n \times n} \rightarrow \C ^{n \times n}$ is the identity map on $n \times n$ matrices. The triple, $(\hat{A}, \hat{b}, \hat{c})_Y$ is called a (descriptor) \emph{matrix-centre realization of $f$ about $Y$}. 

More generally, a descriptor \emph{matrix-centre realization centred at $Y \in \cdm$}, is defined as a triple,
$(A,b,c)_Y$ where, $A := \bsm A_1 \\ \vdots \\ A_d \esm$, is a column $d-$tuple of linear maps, $A_j : \C ^{m\times m} \rightarrow \scr{B} (\cH)$, $b,c \in \scr{B} (\C^m , \cH)$ and $\cH$ is a complex, separable or finite--dimensional Hilbert space. The \emph{linear pencil} of $A$ is a monic, affine-linear function on the $d-$dimensional NC universe over $\C ^{m \times m}$, 
$$ \C ^{(m\N  \times m\N ) \cdot d} := \bigsqcup _{n=1} ^\infty \C ^{(mn \times mn)\cdot d}, $$ defined by
$$ L_A (X) := I_\cH \otimes I_{mn} - \sum _{j=1} ^d ( A_j \otimes \mr{id} _n ) (X); \quad \quad X \in \C ^{(mn \times mn)\cdot d}. $$ 

A priori, $(A,b,c)_Y$ has no dependence on the point $Y \in \cdm$. However, when we say that the matrix-centre realization $(A,b,c)_Y$ is centred at $Y$, we mean that we are considering the $Y-$invertibility domain of $A$, $\scr{D} ^Y(A) \subseteq \C ^{(m\N \times m\N )\cdot d}$, 
$$ \scr{D} ^Y (A) 
    := \bigsqcup _{n=1} ^\infty \scr{D} ^Y _n (A), \quad \quad 
    \scr{D} ^Y _n (A) := \left\{ \left. X \in \C ^{(m n \times m n)\cdot d} \right| 
    \ L_A (X -Y \otimes I_n) ^{-1} \, \text{ exists } \right\}, $$
and the corresponding \emph{quantized transfer function}, $f \sim _Y (A,b,c)$, defined on $\scr{D} ^Y (A)$ via the formula
$$ f (X) 
    = b^* \otimes I_n L_A (X - Y \otimes I_n ) ^{-1} c \otimes I_n; \quad \quad X \in \scr{D} ^Y _n (A). $$
As before, it is readily checked that $\scr{D} ^Y (A)$ is a uniformly open NC set, and that $r \cdot \B ^d _{m \N} (Y) \subseteq \scr{D} ^Y (A)$, for $r := \| A \| _{\mr{col}} ^{-1}$, where $ \| A \| _{\mr{col}}$ denotes the operator norm of the linear map, 
$A = \bsm A_1 \\ \vds \\ A_d \esm : \cdm \rightarrow \scr{B} (\cH)$, defined by 
\[
    A (X) = \bsm A_1 \\ \vds \\ A_d \esm \circ ( X_1, \cdots, X_d) := \sum _{j=1} ^d A_j (X_j); \quad \quad X_j \in \C ^{m \times m}.
\]
\begin{remark}
Any linear map $L: \C ^{m \times m} \rightarrow \scr{B} (\cH)$ is bounded, hence by \cite[Theorem 18]{Paulsen-cb} it is completely bounded by $m \sqrt{m} \| L \| _{\scr{B} (\C ^{m \times m}, \scr{B} (\cH))}$. It follows that we could instead consider the completely bounded norm of the linear map $A : \cdm \rightarrow \scr{B} (\cH)$, but this will not be needed for our investigations here.
\end{remark}

Given an arbitrary matrix-centre realization, $(A,b,c)_Y$, centred at a non-scalar point, $Y \in \cdm$, $m>1$ its quantized transfer function, $f \sim _Y (A,b,c)$, need not be an NC function. 
The $Y-$invertibility domain, $\scr{D} ^Y (A)$, is always a uniformly open NC set, and $f : \scr{D} ^Y (A) \rightarrow \C ^{\N \times \N}$ is graded and respects direct sums, by construction, but it does not necessarily preserve joint similarities. See \cite[Lemma 2.2, Theorem 2.6]{PV1} and \cite[Theorem 6.1]{KMT-matrixreal} for a full characterization of when the quantized transfer function of an arbitrary matrix-centre realization, $f \sim _Y (A,b,c)$, respects joint similarities and is hence an NC function. If $f \sim _Y (A,b,c)$ is an NC function, then it is necessarily uniformly locally bounded on the uniformly open NC set, $\scr{D} ^Y (A)$, and is hence uniformly analytic.  

We will also need to consider matrix-centre realizations of square matrices over $\scr{O} _d (r)$, $r>0$. Let $\Om \subseteq \ncu$ be a uniformly open NC set and let $\scr{O} _d (\Om )$ denote the ring of uniformly analytic NC functions in $\Om$. If $F \in \scr{O} _d (\Om ) ^{\ell \times \ell}$ and $Y \in \Om _m$, then each entry, $f_{i,j}$, of $F$ has a matrix-centre realization at $Y$, $f_{i,j} \sim _Y (A^{(i,j)}, b^{(i,j)}, c^{(i,j)})$ by \cite[Corollary 8.11]{KMT-matrixreal}, where we can assume that each $A^{(i,j)} \in \scr{B} (\cH) ^d$ for $1 \leq i,j \leq \ell$. A matrix-centre realization for $F$ at $Y \in \Om _m$, $F \sim _Y (\hat{A},\hat{b}, \hat{c})$ can then be obtained by setting: $\hat{A} = \bigoplus _{i,j=1} ^\ell A^{(i,j)}$, \emph{i.e.} \footnotesize
\ba \hat{A}_k (G)  & := & \bsm A_k ^{(1,1)} (G) & & & & & & & & & & & \\ & A_k ^{(1,2)} (G) & & & & & & & &  & & \\
& & \ddots & & & & & & & & & \\ & & & A_k ^{(1,\ell)} (G) & & & & & & & & \\ & & & & A_k ^{(2,1)} (G) & & & & & & & \\ & & & & & \ddots & & & & & \\ & & & & & & A_k ^{(2,\ell)} (G) & & & & & \\
& & & & & & & \ddots  & & & & \\
 & & & & & & & & \ \ \ddots   & & & \\
& & & & & & & & & A_k ^{(\ell, 1)} (G) & & \\
& & & & & & & & & & \ddots  & \\ & & & & & & & & & & & A_k ^{(\ell, \ell)} (G) \esm, \quad \hat{b}  :=  \bsm b ^{(1,1)}  & & \\ & \ddots & \\ & & b ^{(1,\ell)}  \\ b ^{(2,1)} & & \\ & \ddots & \\ & & b ^{(2,\ell)} \\ & & & \\ \vdots & \vdots & \vdots & \\ & & & \\ b ^{(\ell ,1 )} & & \\ & \ddots & \\ & & b ^{(\ell,\ell)}  \esm, \ea \normalsize and similarly for $\hat{c}$.

Finally, it will be convenient to consider Fornasini--Marchesini (FM) realizations of matrices over $\scr{O} _d (r)$. As described in \cite[Section 7]{KMT-matrixreal}, if $f \sim _Y (A,b,c)$ is a matrix-centre realization of the quantized transfer function, $f$, $Y \in \cdm$, then $f$ has the FM realization at $Y$, $f \sim _Y (A',B',C',D')$, defined as follows,
\[
    A' _j (G) := A_j , \quad B' _j (G) := A_j (G) c, 
    \quad C' = b^*, \quad \mbox{and} \quad D' := f(Y),
\]
$1 \leq j \leq d$, where $G \in \C ^{m \times m}$.
The quantized transfer function, $f \sim _Y (A', B', C', D')$, is then also given by the FM realization formula, 
\be \label{FMreal} f(X) = D' \otimes I_n + C' \otimes I_n L_{A'} (X - Y \otimes I_n) ^{-1} B' (X - Y \otimes I_n), \quad \quad X \in \scr{D} ^Y _n (A'). \ee Generally, an FM realization about the matrix-centre, $Y \in \cdm$, is a quadruple, $(A,B,C,D) _Y$, where $A_j : \C ^{m \times m} \rightarrow \scr{B} (\cH)$ is a linear map, $B_j : \C ^{m \times m} \rightarrow \cH$ is linear, $1 \leq j \leq d$, $C \in \scr{B} (\cH, \C ^m)$, and $D \in \C ^{m \times m}$. One defines the $Y-$invertibility domain of $A$, $\scr{D} ^Y (A)$, as above, and the realization $(A,B,C,D) _Y$ defines a quantized transfer function on the uniformly open NC set, $\scr{D} ^Y (A)$ as in Equation (\ref{FMreal}) above.

Hence, if $F = ( f_{i,j} ) \in \scr{O} (\Om ) ^{\ell \times \ell}$ where each entry has an FM matrix-cenre realization about $Y \in \Om _m$, $f_{i,j} \sim _Y (A^{(i,j)}, B^{(i,j)}, C^{(i,j)}, D^{(i,j)})$, where $\scr{D} ^Y (A^{(i,j)} ) \supseteq r_{i,j} \cdot \rball (Y)$ for some $r_{i,j} >0$, and we can assume that $A^{(i,j)} _k : \C ^{m\times m} \rightarrow \scr{B} (\cH)$, $B^{(i,j)} _k : \C ^{m\times m} \rightarrow  \cH$, $C ^{(i,j)}: \cH \rightarrow \C ^m$, and $D ^{(i,j)} = f_{i,j} (Y) \in \C ^{m\times m}$ for each $1\leq k\leq d$ and each $1 \leq i,j \leq \ell$. (That is, we assume all $\ell^2$ realizations are defined on the same Hilbert space, $\cH$.) 
An FM matrix-centre realization for $F \in \scr{O} _d (\Om) ^{\ell \times \ell}$ about $Y$ is then  $F \sim _Y (A,B,C,D) $, where $A_k : \C ^{m\times m} \rightarrow \scr{B} (\cH) \otimes \C ^{\ell ^2 \times \ell ^2}$, $B_k : \C ^{m\times m} \rightarrow \cH \otimes \C ^{\ell ^2 \times \ell}$, $C \in \scr{B} (\cH, \C ^m) \otimes \C ^{\ell \times \ell ^2}$ and $D \in \C ^{\ell \times \ell} \otimes \C ^{m \times m}$ are defined by $A _k := \bigoplus _{i,j =1} ^\ell A_k ^{(i,j)}$, \small
\ba
 B_k (G) & := & \bsm B_k ^{(1,1)} (G) & & \\ & \ddots & \\ & & B_k ^{(1,\ell)} (G) \\ B_k ^{(2,1)} (G) & & \\ & \ddots & \\ & & B_k ^{(2,\ell)} (G) \\ & & & \\ \vdots & \vdots & \vdots & \\ & & & \\ B_k ^{(\ell ,1 )} (G) & & \\ & \ddots & \\ & & B_k ^{(\ell,\ell)} (G) \esm, 
 \quad \quad C := \bsm C ^{(1,1)} & \cdots & C^{(1,\ell)} & & & & & & & \\
& & & C ^{(2,1)} & \cdots & C^{(2,\ell)} & & & & \\
& & &  &  & & \ddots & & & \\
& & & & & & & C ^{(\ell,1)} & \cdots & C^{(\ell,\ell)}  \esm, \ea \normalsize $G \in \C ^{m \times m}$, and $D := F (Y) \in \C ^{m\times m} \otimes \C ^{\ell \times \ell}$. By construction, $\scr{D} ^Y (A) \supseteq r \cdot \rball (Y)$, where $r := \min _{1\leq i,j \leq \ell}  r _{i,j}$.  

If $F$ is an $\ell \times \ell$ matrix of quantized transfer functions, and $F \sim _Y (A,B,C,D)$ is invertible at $Y \in \cdm$, \emph{i.e.} $D=F(Y) \in \C ^{\ell m \times \ell m}$ is invertible, then $F^{-1} \sim _Y (A^{(-1)}, B^{(-1)}, C^{(-1)}, D^{-1} )$  also has a realization centred at $Y$ with 
$$ A^{(-1)} _j (G) := A_j (G) - B_j (G) D^{-1} C, \quad \quad B_j ^{(-1)} (G) := -B_j (G) D^{-1}; \quad \quad G \in \C ^{m \times m}, $$ $C^{(-1)} = D^{-1} C$, and $D^{-1} = F(Y) ^{-1}$, see, for example, \cite[Section 7.1]{KMT-matrixreal}.

\section{Realizations for $\scr{O} _d (r)$} \label{sec:odrreal}

From a certain perspective, the invertibility of $I - A\otimes X$ should depend on the ``norm" of $A$ in some quantity.
More specifically, in the single variable setting, it is readily seen that $I - A\otimes X$ is invertible in a ball of radius $1/r_\sigma(A)$, where $r_\sigma(A)$ is the spectral radius of $A$. This rather intuitive idea extends to the non-commutative setting.

\begin{defn}
Given $A = \bsm A_1 \\ \vds \\ A_d \esm \in \scr{B}(\cH)^d$, let $\mr{Ad}_{A^*,A}:\scr{B}(\cH) \to \scr{B}(\cH)$ denote the linear map of \emph{adjunction} by $A^*$ and $A$, defined by $\mr{Ad}_{A^* ,A}(X) = \sum_i A_i ^* X A_i$.

The (column) \emph{joint spectral radius} of $A$, $r_\sigma(A)$, is 
\[ r_{\sigma}(A) = \lim_{k\to\infty} \norm{\mr{Ad}_{A^*,A}^{\circ k}(I_{\cH})}^{\frac{1}{2k}}
    \]
where $\mr{Ad}_{A^*,A}^{\circ k}$ denotes the composition of $\mr{Ad}_{A^*,A}$ with itself $k$ times.   
\end{defn}

The limit in the above definition exists since $\norm{\cdot}$ is sub-multiplicative. 
The peculiar choice of notation for $\mr{Ad}_{A^*,A}$ is because, more generally, we can define a similar map $\mr{Ad}_{A,C}$ for tuples $A \in \scr{B}(\cH)^d$ and $C \in \scr{B}(\cK)^d$.

\begin{prop} \label{odrreal}
Any $f \in \fpsd$ belongs to $\scr{O} _d (r)$ if and only if it has a realization, $f \sim (A,b,c) \in \scr{B} (\cH) ^d \times \cH \times \cH$, with $r_\sigma (A) \leq \frac{1}{r}$.
\end{prop}
\begin{proof}
First, assume that $f \sim (A,b,c)$ with $r_\sigma (A) \leq r^{-1}$. Then, 
\ba \frac{1}{R_f} & = & \lim _{n\rightarrow \infty} \sqrt[2n]{\sum _{|\om| = n} c^* A^{\om ^*} b \, b^* A^\om c} \\
& \leq & \lim_{n\to\infty} \sqrt[2n]{\| \mr{Ad} _{A^*,A} ^{\circ n} (I) \|} \ \| b \| ^{\frac{1}{n}} \|c \| ^{\frac{1}{n}} \\
& = & r_\sigma (A) \leq \frac{1}{r}, \ea so that $r \leq R_f$.  

Conversely, consider 
$$ f (\fz ) = \sum _{\om \in \F} \hat{f} _\om \fz ^\om = \sum _{n=0} ^\infty f _n (\fz ). $$  where each $f_n \in \fp$ is a homogeneous free polynomial of degree $n$. Hence, each $f_n$ has the backward free shift realization, $f_n ^\mrt \sim (L^*, 1, f _n ^\mrt)$, where $L = (L_1, \cdots, L _d )$ is the $d-$tuple of left free shifts on the full Fock space or free Hardy space, $\hardy$, (see \cite[Subsection 2.3, Lemma 3.2]{AMS-opreal}) and $f^\mrt$ denotes the \emph{transpose} of any free FPS, defined by 
$$ f ^\mrt (\fz ) = \sum _{\om \in \F} \hat{f} _{\om ^\mrt} \fz ^\om = \sum _{n=0} ^\infty f ^\mrt _n (\fz ), $$ where for any word $\om = w_1 \cdots w_m \in \F$, $\om ^\mrt := w_m \cdots w_1$. Hence, each $f_n ^\mrt$ is also a homogeneous free polynomial of degree $n$ (unless it is identically $0$). Let $P_n$ be the projection onto the (finite--dimensional) controllable subspace, 
$$ \scr{C} _n := \bigvee _{\om \in \F} L^{*\om} f_n ^\mrt. $$ 
Hence, each $f_n$ has the finite--dimensional realization $f _n  \sim (A^{'(n)}, b_n', c_n') := (L^* | _{\scr{C} _n}, 1, f_n ^\mrt)$, obtained by compressing $(L^*, 1, f_n ^\mrt)$ to $\scr{C} _n$. (Since $f_n ^\mrt$ is a polynomial, note that $1 \in \scr{C} _n$.) This realization is controllable, \emph{i.e.} $f_n ^\mrt$ generates $\scr{C} _n$ as an $L^*-$cyclic vector, but it is also minimal in the sense that $1$ is cyclic for the adjoints of $L^* | _{\scr{C} _n}$. Indeed, the observable subspace is defined by 
$$ \bigvee (P_n L P_n)^\om P_n 1 = P_n \bigvee L^\om 1 = \scr{C} _n. $$  It follows that $(A^{'(n)}, b_n', c_n')$ is a finite, minimal realization of a homogeneous free polynomial, and hence it is jointly nilpotent of order $n+1$, \emph{i.e.} $A^{'(n) \om} \equiv 0$ for any $|\om | \geq n +1$.

Consider the similar realization, $f_n  \sim (A ^{(n)}, b_n, c _n)$, defined by $A^{(0)} =0$, and for any $n \in \N$,
$$ A ^{(n)} := \sqrt[n]{n^2} \sqrt[n]{\| f_n  \| _{\hardy}} A^{'(n)}, $$ 
$$ b _0 := 1, \quad b _n := \frac{1}{n} b'_n, $$ 
$$ c _0 := \hat{f} _\emptyset \quad \mbox{and} \quad c _n := \frac{c'_n}{\| f_n \| _{\hardy}} \cdot \frac{1}{n}. $$ 
Note, in the above that $\| f ^\mrt \| _{\hardy} = \| f \| _{\hardy}$ for any $f \in \hardy$.

By construction, $\scr{C} _n$, is $L^*-$invariant, so that 
\ba \| A^{'(n)} \| ^2 _{\mr{col}} & = & \left\| \sum _{j=1} ^d A^{'(n) *} _j A^{'(n)} \right\| \\
& = & \left\| \sum P_n L_j P_n L_j ^* P_n \right\| = \left\| P_n \sum L_j L_j ^* P_n \right\| \\
& \leq & \| L L^* \| = \| L \| _{\mr{row}} ^2 = 1. \ea 
We now define 
$$ \cH := \C \oplus \bigoplus _{n=1} ^\infty \scr{C} _n, $$
$$ A := 0 \oplus \bigoplus _{n=1} ^\infty A^{(n)},  \quad b:= \bigoplus _{n=0} ^\infty b_n \quad \mbox{and} \quad  c:= \bigoplus_{n=0}^\infty c_n. $$ Hence, $f \sim (A,b,c)$ is a realization of $f$, and 
\ba \left\| \sum _{|\om | = n } A ^{\om ^*} A^{\om } \right\| & = & \| \mr{Ad} _{A^*,A} ^{\circ n} (I) \| \\
& \leq & \left\| A - \bigoplus_{k=1}^{n-1}A^{(k)}\right\| _{\mr{col}} ^{2n} \leq \left( \sup _{m \geq n}   \sqrt[m]{m^2} \sqrt[m]{\| f_m\| _{\hardy}} \right)^{2n}, \ea since each $A^{(n)}$ is jointly nilpotent of order $n+1$. We conclude that 
\ba r _\sigma (A) & = & \lim _{n\rightarrow \infty} \sqrt[2n]{\left\| \mr{Ad} _{A^*,A} ^{\circ n} (I) \right\|} \\
& = & \lim _{n\rightarrow \infty} \sup _{m \geq n} \sqrt[m]{m^2} \sqrt[2m]{\| f _m \| _{\hardy} ^2} \\
& = & \limsup _{n\rightarrow \infty} \sqrt[n]{n^2} \sqrt[2n]{\| f_n \| _{\hardy} ^2} \\
& = & \limsup \sqrt[2n]{\sum _{|\om| =n} | \hat{f} _\om | ^2 } = \frac{1}{R_f} \leq \frac{1}{r}. \ea 
\end{proof}

By the non-commutative Rota theorem of \cite[Theorem 4.8]{Pop-joint}, given a column $d-$tuple, $A \in \scr{B} (\cH) ^d$ and any $\eps >0$, $A$ is jointly similar to a column $d-$tuple $A' \in \scr{B} (\cH) ^d$ with $\| A ' \| _{\mr{col}} < r_\sigma (A) + \eps$. 

\begin{cor}
Let $F \in \scr{O} _d (r) ^{\ell \times \ell}$. Then, for any $Y \in r \cdot \B ^{(n\times n) \cdot d}$, $F$ has a minimal FM matrix centre realization about $Y$, $F \sim _Y (A,B,C,D)$ so that $\scr{D} ^Y (A) \supseteq r \cdot \B ^{(\N n \times \N n ) \cdot d} = r \cdot \rball \cap \ncun$. 
\end{cor}
\begin{proof}
Each matrix-entry, $f_{i,j}$, of $F$ belongs to $\scr{O} _d (r)$, and hence has radius of convergence at least $r>0$. By the previous proposition and the NC Rota theorem \cite[Theorem 4.8]{Pop-joint}, given any fixed $s>r$, each matrix entry, $f_{i,j}$, of $F$ has a minimal (descriptor, and hence) FM realization (about $0$), $f_{i,j} \sim (A^{'(i,j)}, B^{'(i,j)}, C^{'(i,j)}, D^{'(i,j)})$, with $\| A^{'(i,j)} \| _{\mr{col}} \leq \frac{1}{s}$, so that $\scr{D} (A^{'(i,j)}) \supseteq s \cdot \rball$ for all $1 \leq i,j \leq \ell$ and any $s>r$. Hence, by \cite[Section 5.2, Theorem 5.2]{KMT-matrixreal}, we can translate these FM realizations from $0$ to $Y$ to obtain minimal FM matrix-centre realizations of each $f_{i,j}$ about $Y$, $f_{i,j} \sim _Y  (A^{(i,j)}, B^{(i,j)}, C^{(i,j)}, D^{(i,j)})$, with $$ \scr{D} ^Y (A^{(i,j)}) = \scr{D} (A^{'(i,j)}) \cap \ncun \supseteq s \cdot \B ^{(\N n \times \N n ) \cdot d}. $$ Hence, as described in the previous section, we can construct a minimal FM matrix-centre realization for the entire matrix $F$ about $Y$, $F \sim _Y (A,B,C,D)$, with $\scr{D} ^Y (A) \supseteq  s \cdot \B ^{(\N n \times \N n ) \cdot d}$. Since $s >r$ is arbitrary, we conclude that $\scr{D} ^Y (A) \supseteq  r \cdot \B ^{(\N n \times \N n ) \cdot d}$.
\end{proof}

We will need the following fact in the proofs of some of our main results.
\begin{prop} \label{realatY}
Let $H \in \scr{O} _d (r) ^{\ell \times \ell}$ be full and $\Xi \in \mf{M} ^{1\times d}$ obey $\| \Xi \| _{\mr{row}} <r$. Then there exists $Y \in r \cdot \bdn$ and a matrix-centre realization of $H$ about $Y$, $H \sim _Y (A,B,C,D)$, $D = H(Y)$, so that $D$ is invertible and $\| A \| _{\mr{col}} \cdot ( \| \Xi \| _{\mr{row}} + \| Y \| _{\mr{row}} ) <1$. 
\end{prop}

\begin{proof}
Each matrix-entry of $H$ has a realization at $0$, $H_{i,j} \sim _0 (a^{(i,j)}, b^{(i,j)}, c^{(i,j)})$, where we can assume that the joint (column) spectral radius of each $a^{(i,j)} \in \scr{B} (\cH) ^d$ is at most $\frac{1}{r}$, for $1\leq i,j \leq \ell$. Hence, given $\Xi \in \mf{M} ^{1\times d}$ with $\| \Xi \| _{\mr{row}} <r$, we can choose each $a^{(i,j)}$ so that $\max _{i,j} \| a ^{(i,j)} \| _{\mr{col}} \cdot \| \Xi \| _{\mr{row}} <1$. A realization for $H$ at $0$, $H \sim _0 (a,b,c)$ is then constructed by setting $a := \bigoplus _{i,j =1} ^\ell a^{(i,j)} \in \scr{B} (\cH ^{\ell ^2} ) ^{1 \otimes d}$ and defining $b,c \in \scr{B} (\C ^\ell , \cH ^{\ell ^2})$ as described at the end of Subsection \ref{ss:realize}. Hence, $\| a \| _{\mr{col}} \cdot \| \Xi \| _{\mr{row}} <1$.

Again, as described in Subsection \ref{ss:realize}, we can translate the realization of $H$ at $0$ to any point $Y \in r \cdot \bdn$ to obtain an FM realization of $H$ about $Y$, $H \sim _Y (A,B,C,D)$, with $A _j : \C ^{n \times n} \rightarrow \scr{B} (\cH ^{n \ell ^2})$, $B_j : \C ^{n\times n} \rightarrow \cH ^{n \ell ^2}$, $C : \cH ^{n \ell ^2} \rightarrow \C ^{n\ell}$ and $D = H(Y) \in \C ^{n \ell \times n\ell}$, by the formulas
$$ A_j (G) := L_a (Y) ^{-1} a_j \otimes G, \quad B_j (G) := L_a (Y) ^{-1} a_j \otimes G L_a (Y) ^{-1} c\otimes I_n, $$ and $C := b^* \otimes I_n$. By \cite[Lemma 5.8, Corollary 5.9]{AMS-ncentire}, there is an $n \in \N$ so that the set of all $Y \in r \cdot \bdn$ with the property that $H(Y)$ is invertible is analytic-Zariski dense. It follows from the above formulas for a matrix-centre realization of $H$ about any $Y \in r \cdot \bdn$, that if we choose $\| Y \| _{\mr{row}}$ sufficiently small, $L_a (Y)$ will be arbitrarily close to $I _{\cH ^{n \ell ^2}}$ and $\| A \| _{\mr{col}}$ will be arbitrarily close to $\| a \| _{\mr{col}}$. Hence, since $\| a \| _{\mr{col}} \cdot \| \Xi \| _{\mr{row}} <1$, we can choose $Y \in r \cdot \bdn$ sufficiently close to $0$, so that $H(Y)$ is invertible and so that $H \sim _Y (A,B,C,D)$, $D = H(Y)$, with 
$$ \| A \| _{\mr{col}} \cdot ( \| \Xi \| _{\mr{row}} + \| Y \| _{\mr{row}} ) < 1. $$ 
\end{proof}

\section{Singularities of linear pencils and NC functions} \label{sec:pencils}

Given $T \in \cdn$ and $\Xi \in \mf{M} ^{1\times d}$, let $T \otimes \Xi = \sum T_j \otimes \Xi _j$.
Recall, $\cJ := L^2 (\mf{M}, \tau)$.

\begin{prop} \label{Schursing}
Let $\Xi \in \mf{M} ^{1\times d}$ and let $H$ be the $\ell \times \ell$ matrix-valued quantized transfer function of a matrix-centre realization at $Y \in \cdn$, $H \sim _Y (A,B,C,D)$, so that $D=H(Y)$ is invertible. If $T \in \cdn$ is such that $L_A (T \otimes \Xi - Y \otimes I_\cJ)$ is invertible, then $H (T \otimes \Xi)$ has non-trivial kernel or co-kernel if and only if $L _{A^{(-1)}} (T \otimes \Xi - Y \otimes I_\cJ)$ has non-trivial kernel or co-kernel, respectively, where $H^{-1} \sim _Y (A^{(-1)}, B^{(-1)}, C^{(-1)}, D^{-1} )$ is the realization of $H^{-1}$ about $Y$. 

Moreover, 
$$ \Phi _T := -D ^{-1} C \otimes I_\cJ  | _{\nbker L_{A^{(-1)}} (T \otimes \Xi - Y \otimes I_\cJ)} : \nbker L_{A^{(-1)}} (T \otimes \Xi - Y \otimes I_\cJ) \twoheadrightarrow \nbker H(T \otimes \Xi), $$ is a linear isomorphism with inverse 
$$ \Phi ^{-1} _T := L_A (T \otimes \Xi - Y \otimes I_\cJ) ^{-1} B ( T \otimes \Xi - Y \otimes I_\cJ) | _{\nbker H(T \otimes \Xi)} : \nbker H(T \otimes \Xi ) \twoheadrightarrow \nbker L_{A^{(-1)}} (T \otimes \Xi - Y \otimes I_\cJ). $$
Similarly, 
$$ \Psi _T := -D^{-*} \otimes I _\cJ  \, B (T \otimes \Xi - Y \otimes I_\cJ) ^* | _{\nbker L_{A^{(-1)}} (T \otimes \Xi - Y \otimes I_\cJ) ^* }: \nbker L_{A^{(-1)}} (T \otimes \Xi - Y \otimes I_\cJ) ^* \twoheadrightarrow \nbker H (T \otimes \Xi ) ^*, $$ is an isomorphism with inverse
$$ \Psi ^{-1} _T := L_A (T \otimes \Xi - Y \otimes I_\cJ) ^{-*}  C^* \otimes I_\cJ | _{ \nbker H(T \otimes \Xi) ^*} : \nbker H(T\otimes \Xi) ^* \twoheadrightarrow \nbker L_{A^{(-1)}} (T \otimes \Xi - Y \otimes I_\cJ) ^*. $$
\end{prop}

In the above proposition statement, recall that we assume that each matrix-entry, $H^{(i,j)}$ of $H$ is a quantized transfer function that has a realization, $H^{(i,j)} \sim _Y (A^{(i,j)}, B^{(i,j)}, C^{(i,j)}, D^{(i,j)})$, where $Y \in \cdn$. 
We assume $A^{(i,j)} _k : \C ^{n\times n} \rightarrow \scr{B} (\cH)$, $1\leq k \leq d$ for a fixed, separable Hilbert space $\cH$, so that $A_k := \oplus _{i,j =1} ^\ell A_k ^{(i,j)} : \C ^{n \times n} \rightarrow \scr{B} (\cH ^{\ell ^2} )$. Similarly, as before $B_k \in \scr{B} (\C^n \otimes \C ^\ell, \cH ^{\ell ^2})$, $C \in \scr{B} (\cH ^{\ell ^2}, \C ^n \otimes \C ^\ell)$ and $D \in \C ^{n\ell \times n \ell}$. 

If $M:= \bsm A & B \\ C & D \esm$ is a block operator matrix with invertible $A,D$, then the Schur complements by $A$ and $D$ are:
$$ M/D := A - B D^{-1} C, \quad \mbox{and} \quad M/A := D - C A^{-1} B. $$
\begin{lemma}{(\cite[Lemma 2.2]{MS-peak})} \label{MarSha}
Let $M := \bsm A & B \\ C & D \esm : \cH \oplus \cJ \rightarrow \cH \oplus \cJ$ be a block operator matrix so that $A \in \scr{B} (\cH)$, $B \in \scr{B} (\cJ , \cH)$, $C \in \scr{B} (\cH , \cJ)$ and $D \in \scr{B} (\cJ)$. If both $A, D$ are invertible, then 
$$ \Phi := D^{-1} C | _{\nbker M/D} : \nbker M/D \twoheadrightarrow \nbker M/A, $$ is an isomorphism with inverse 
$$ \Phi ^{-1} := A^{-1} B | _{\nbker M/A} : \nbker M/A \twoheadrightarrow \nbker M/D. $$
If $A$ is invertible then 
$$ \nbker M = \bpm -A^{-1} B \\ I \epm \nbker M/A, $$ and if $D$ is invertible then
$$ \nbker M = \bpm I \\ -D^{-1} C \epm \nbker M/D. $$ 
\end{lemma}
The above lemma implies, in particular, that if both $A$ and $D$ are invertible then
$$ \nbker M = \bpm I \\ \Phi \epm \nbker M/D = \bpm \Phi ^{-1} \\ I \epm \nbker M/A. $$
\begin{proof}
Only the final statements on the kernel of $M$ need to be proven. The previous statements are proven in \cite[Lemma 2.2]{MS-peak}. 

If $A$ is invertible then
$$ M = \bpm I & 0 \\ C A^{-1} & I \epm \bpm A & 0 \\ 0 & M /A \epm \bpm I & A^{-1} B \\ 0 & I \epm, $$ so that if $x \in \nbker M/A$ then
$$ \bpm I & - A^{-1} B \\ 0 & I \epm \bpm 0 \\ x \epm = \bpm - A^{-1} B x \\ x \epm \in \nbker M. $$ Conversely, if $A$ is invertible and $\bsm x \\ y \esm \in \nbker M$ then 
\ba \bpm 0 \\ 0 \epm & = & \bpm A & 0 \\ 0 & M /A \epm \bpm I & A^{-1} B \\ 0 & I \epm \bpm x \\ y \epm \\
& = & \bpm A & 0 \\ 0 & M /A \epm  \bpm x + A^{-1} B y \\ y \epm. \ea
Hence, $y \in \nbker M/A$ and since $A$ is invertible, $0 = Ax + By$ so that $x= -A^{-1}B y$. 
The final claim on $\nbker M/D$ follows from a symmetric Schur complement argument.
\end{proof}

\begin{proof}[Proof of Proposition \ref{Schursing}]
Consider the block operator--matrix,
$$ M(\Xi) := \bpm L_A (T \otimes \Xi  - Y \otimes I_{ \cJ}) & B (T \otimes \Xi - Y \otimes I_\cJ) \\  C \otimes I_\cJ & -D \otimes I_\cJ \epm. $$
Consider the two Schur complement factorizations of $M(\Xi)$: First, 
$$ M(\Xi)  =  R \cdot \underbrace{\bpm L_A (T \otimes \Xi - Y\otimes I_\cJ) + B(T \otimes \Xi -Y \otimes I_\cJ) D^{-1} C \otimes I _\cJ  & 0 \\ 0 & -D \otimes I _\cJ \epm}_{= \bpm L_{A^{(-1)}} (T \otimes \Xi - Y \otimes I _\cJ) & 0 \\ 0 & -D \otimes I _\cJ \epm } \cdot S , $$
where 
$$ R  := \bpm I_{\cH ^{\ell ^2} \otimes \cJ} & - B(T \otimes \Xi - Y\otimes I _\cJ) D^{-1} \otimes I_\cJ \\ 0 &  I_{\ell n} \otimes I_\cJ  \epm, \quad \mbox{and} \quad S := \bpm I_{\cH ^{\ell^2} \otimes \cJ} & 0 \\  C \otimes I_\cJ & -D \otimes I _\cJ \epm, $$ are invertible matrices, by assumption. Here, we assume that the realization $(A,B,C,D)_Y$ of $F$ about $Y$ is constructed from the $\ell ^2$ minimal FM matrix-centre realizations of each matrix-entry, $f_{i,j}$, of $F$ about $Y$, $(A^{(i,j)}, B^{(i,j)}, C^{(i,j)}, D^{(i,j)} )_Y$, with each of these $\ell ^2$ realizations defined in terms of the same complex, separable (or finite) Hilbert space $\cH$, \emph{i.e.} $A^{(i,j)} _k : \C ^{n\times n} \rightarrow \scr{B} (\cH)$, so that $A_k : \C ^{n \times n} \rightarrow \scr{B} ( \cH \otimes \C ^{\ell ^2}) \simeq \scr{B} (\cH) \otimes \C ^{\ell ^2 \times \ell ^2}$. In the above, we have written $\cH ^{\ell ^2}$ in place of $\cH \otimes \C ^{\ell ^2}$ to simplify notation. Also, in the above, since $D= H(Y)$ is invertible, $H^{-1} \sim _Y (A^{(-1)}, B^{(-1)}, C^{(-1)}, D^{-1})$ where recall that 
$$ A^{(-1)} _j (G) = A_j (G) - B_j (G) D^{-1} C. $$ 

The second Schur complement factorization is then
$$M(\Xi ) = R' \underbrace{\bpm L_A (T \otimes \Xi - Y\otimes I _\cJ)  & 0 \\ 0 & - D \otimes I_\cJ -  C \otimes I_\cJ  L_A ( T \otimes \Xi   - Y \otimes I_\cJ ) ^{-1} B (T \otimes \Xi  - Y\otimes I_\cJ) \epm}_{= \bpm  L_A (T \otimes \Xi - Y\otimes I_\cJ)  & 0 \\ 0 & - H(T \otimes \Xi) \epm } S', $$ 
where 
$$ R' :=  \bsm I_{\cH ^{\ell ^2} \otimes \cJ} & 0 \\  C \otimes I_\cJ  L_A ( T \otimes \Xi  - Y\otimes I_\cJ) ^{-1} & I_{\ell n} \otimes I_\cJ \esm  \quad \mbox{and} \quad S' := \bsm I_{\cH ^{\ell ^2} \otimes I_\cJ} & L_A (T \otimes \Xi - Y \otimes I_\cJ ) ^{-1} B (T \otimes \Xi - Y\otimes I_\cJ) \\  0 & I _{\ell n} \otimes I_\cJ \esm, $$ are again, both invertible. 

It follows that $H (\Xi)$ has non-trivial kernel/co-kernel if and only if $L_{A^{(-1)}} (T \otimes \Xi - Y \otimes I)$ does. The remaining statements follow from Lemma \ref{MarSha} applied to $M (\Xi)$ and $M(\Xi) ^*$ using that if $M= \bsm A & B \\ C & D \esm$ with invertible $A$ and $D$, then $(M/A) ^* = M^* / A^*$ and $(M/D) ^* = M^* / D^*$. 
\end{proof}

\section{Singular linear pencils and the relation co-dimension} \label{sec:sing_lin_pencils}

Here we extend some of the results of \cite[Section 3]{freefield} to infinite--dimensional realizations. Let $A := \bsm A _0 \\ A_1 \\ \vds \\ A_d \esm \in \scr{B} (\cH) ^{d+1}$. The non-monic, \emph{affine linear pencil} of $A$ is:
$$ \hat{A} (Z) := I_n \otimes A_0 + Z \otimes A := I_n \otimes A _0 + \sum _{j=1} ^d Z_j \otimes A_j; \quad \quad Z \in \cdn. $$ Let $(h_i)$ be a fixed orthonormal basis of the complex and separable Hilbert space, $\cH$, let $\TC (\cH)$ denote the trace-class operators on $\cH$ and let $\HS (\cH)$ denote the Hilbert--Schmidt operators. 

\begin{lemma} \label{kercomm} Let $A \in \scr{B} (\cH) ^{d+1}$.
Suppose that $e := \sum _i h_i \otimes e_i$, $f:= \sum _j h_j \otimes f_j$, $e,f \in \cH \otimes \cJ$ obey $\hat{A} (\Xi) f =0 = \hat{A} (\Xi) ^* e$, where $\Xi = ( \Xi _1, \cdots, \Xi _d ) \in \mf{M} ^{1 \times d}$ is arbitrary. Define
$$ T_k := \sum _{i,j} A^{(k)} _{i,j} \ip{e_i}{\cdot} _{\cJ} f_j; \quad \quad 0 \leq k \leq d, \quad A^{(k)} _{i,j} := \ip{h_i}{A^{(k)} h_j}_{\cH}. $$ Then $T_k \in \TC (\cJ)$ obey
$$ \sum _{k=1} ^d [T_k, \Xi _k] =0. $$ 
\end{lemma}
\begin{proof}
Fix $0 \leq k \leq d$ and set $A:= A^{(k)}$, $T := T_k$ as in the lemma statement. Choose any orthonormal basis, $(g_j)$ of $\cJ$ and define the unitary $U : \cJ \twoheadrightarrow \cH$ by $U g_j := h_j$. Further define $S_f, S_e : \cH \rightarrow \cJ$ by $S_f h_i := f_i$ and $S_e h_i := e_i$. Finally, set $L_f := S_f U$ and $L_e := S_e U$ so that $L_f, L_e : \cJ \rightarrow \cJ$. Then,
\ba \mr{tr} _\cJ (L_f ^* L_f) & = & \sum _i \ip{g_i}{L_f ^* L_f g_i}_{\cJ} = \sum _i \ip{f_i}{f_i}_{\cJ} \\
& = & \sum _i \| f_i \| ^2 _\cJ = \| f \| ^2 _{\cH \otimes \cJ} < +\infty, \ea so that $L_f$, and similarly $L_e$, are Hilbert--Schmidt operators on $\cJ$. In the above, $\mr{tr} _\cJ$ denotes the faithful trace over $\cJ$. Now observe that 
\ba L_f U^* A U L_e ^* & = & S_f A S_e ^*  \\
& = & S_f \circ \sum _{i}  \ip{h_i}{\cdot}_\cH h_i \circ A \circ \sum _j \ip{h_j}{\cdot}_\cH h_j \circ S_e ^*  \\ 
& = & \sum _{i,j} S_f h_i \ip{h_i}{A h_j} _\cH \ip{S_e h_j}{\cdot}_\cJ \\
& = & \sum _{i,j} A_{i,j} \ip{e_j}{\cdot}_\cJ f_i =: \wt{T}. \ea 
Since both $L_f, L_e$ are Hilbert--Schmidt, it follows that $\wt{T}$ is necessarily trace--class \cite[Theorem VI.22(h)]{RnS}. Moreover, the transpose of $\wt{T}$ with respect to the orthonormal basis $(h_i)$,
$$ \sum _{i,j} A^\mrt _{i,j} \ip{e_j}{\cdot}_\cJ f_i = \sum _{i,j} \ip{h_j}{A h_i}_\cH \ip{e_j}{\cdot}_\cJ f_i = T, $$ is then also trace-class.

Now $\hat{A} (\Xi) f =0$ means that 
$$ 0  =  \sum _j A^{(0)} _{i,j} f_j + \sum _{k=1} ^d \sum _j A^{(k)} _{i,j} \Xi _k f_j, $$ and hence
\ba 0 & = & \sum _{i,j} \ip{e_i}{\cdot}_\cJ A^{(0)} _{i,j} f_j + \sum _{k=1} ^d \sum _{i,j} \ip{e_i}{\cdot}_\cJ A^{(k)} _{i,j} \Xi _k f_j \\
& = & T_0 + \sum _{k=1} ^d  \Xi _k T_k. \ea 
Similarly, $\hat{A} (\Xi) ^* e =0$ implies that 
$$ T_0 + \sum _{k=1} ^d T_k \Xi _k =0, $$ and we conclude that 
$$ \sum _{k=1} ^d [T_k, \Xi _k] =0, $$ as desired. 
\end{proof}

\subsection{Gram matrix calculations} 

Let $(f_i) _{i=1} ^\infty \subseteq \cJ$ be such that $\sum _i \| f_i \| ^2 _\cJ < +\infty$, \emph{i.e.} $\sum _i h_i \otimes f_i \in \cH \otimes \cJ$ where $(h_i)$ is any fixed orthonormal basis of $\cH$. 

\begin{defn}
A square--summable sequence, $(f_i) _{i=1} ^\infty \subseteq \cJ$, $\sum \| f_i \| _\cJ ^2 < +\infty$, is \emph{linearly independent} if: Given any $\mbf{c} = (c_i) \in \ell ^2$, $\sum c_i f_i =0 \in \cJ$ implies that $\mbf{c} =0$. 
\end{defn}

Given a square--summable sequence, $\mbf{f} := (f_j) _{j=1} ^\infty \subseteq \cJ$,  let 
$$ f:= \sum _j h_j \otimes f_j \in \cH \otimes \cJ,  $$ where $(h_j) _{j=1} ^\infty$ is the fixed orthonormal basis of $\cH$.

\begin{lemma} \label{linindep} 
Let $A \in \scr{B} (\cH)$. If $\mbf{f} = (f_j) \subseteq \cJ$ is a square--summable sequence that is linearly independent, $f = \sum _j h_j \otimes f_j \in \cH \otimes \cJ$, and $A \otimes I_\cJ f =0$, then $A =0$. 
\end{lemma}
\begin{proof}
If $A \neq 0$ and $A \otimes I _\cJ f =0$, then, expanding $A$ in the orthonormal basis of $\cH$, 
$$ A = \sum _{j,k} \underbrace{\ip{h_j}{A h_k} _\cH}_{=: A_{j,k}} \ip{h_j}{\cdot}_\cH h_k. $$ 
Hence, 
\ba0 & = & A \otimes I_\cJ f  =  \sum _{i,j,k} A_{j,k} \ip{h_j}{h_i} h_k  \otimes f_i \\
& = & \sum _{j,k} A_{j,k} h_k \otimes f_j \\
& = & \sum _k h_k \otimes \sum _j A_{j,k} f_j. \ea
That is, for each fixed $k \in \N$, 
$$ \sum _j A_{j,k} f_j =0, $$ so that $A_{j,k} =0$ for all $j,k \in \N$, since $(f_i)$ is linearly independent. That is, $A =0$. 
\end{proof}

Given a square--summable sequence, $\mbf{f} = (f_j ) \subseteq \cJ$, or equvalently $f = \sum _j h_j \otimes f_j \in \cH \otimes \cJ$, we define the \emph{Gram matrix}, $G_{f} \in \scr{B} (\cH)$ of $f$ as the linear operator
$$ G_f = \sum _{i,j} \ip{f_i}{f_j} _{\cJ} \ip{h_j}{\cdot}_\cH h_i, $$ where $(h_j)$ is a fixed orthonormal basis of $\cH$. Note that $G_{\mbf{f}}$ is defined so that 
$$  (G _f)_{i,j} = \ip{f_i}{f_j}_\cJ. $$ We then define the operator 
$$ \Ga _f := G_f ^\mrt = \sum _{i,j} \ip{f_i}{f_j} _\cJ \ip{h_i}{\cdot}_\cH h_j, $$ which has matrix entries $(\Ga _f) _{i,j} = \ip{f_j}{f_i}_\cJ$, so that $\Ga _f$ is the transpose of $G_f$ with respect to the basis $(h_k)$. (That is, $\mrt$, denotes transpose.) Alternatively, $\ov{G}_f := C G_f C = \Ga _f$, where $\ov{G} _f$ denotes entry-wise complex conjugation of $G_f$. That is, we can define a conjugation of $\cH$, $C : \cH \rightarrow \cH$, by complex conjugation with respect to the standard basis of $\cH$. Namely, if $h = \sum c_j h_j \in \cH$, then $Ch := \sum \ov{c_j} h_j$ and $C$ is an anti-linear, isometric involution of $\cH$. We will call $\Ga _f$ the \emph{Gramian of $f$}. 

\begin{lemma}
Let $(f_j)$ be a square--summable sequence in $\cJ$. Then the Gramian,  $$\Ga _f := \sum _{i,j} \ip{f_i}{f_j}_\cJ \ip{h_i}{\cdot}_{\cH} h_j \in \TC (\cH),$$ and the Gram matrix, $G_f = \Ga _f ^\mrt$, are positive semi-definite trace--class operators. 
\end{lemma}
\begin{proof}
Clearly, given any $h = \sum _i c_i h_i \in \cH$,
\ba \ip{h}{\Ga _f ^\mrt  h} _{\cH} & = & \sum _{i,j} \ov{c} _i \ip{f_i}{f_j}_\cJ c_j  \\
& = & \ip{ \sum _i c_i f_i}{\sum _j c_j f_j}_\cJ = \ip{g}{g}_\cJ, \ea where $g:= \sum c_i f_i \in \cJ$ since 
$$ \| \sum c_i f_i \| \leq \sum _i |c_i | \| f_i \| _\cJ \leq \| h \| _\cH \sqrt{ \sum _i \| f_i \| ^2 _\cJ } < +\infty. $$
Hence, $\Ga _f ^\mrt = G_f$ is bounded and since $\ip{g}{g}_\cJ \geq 0$, it follows that $\Ga _f ^\mrt \geq 0$. Also,
$$ \mr{tr} \, \Ga _f ^\mrt = \sum _i \ip{f_i}{f_i} _\cJ < +\infty, $$ so that $\Ga _f ^\mrt \in \TC (\cH)$.
It follows that $\Ga _f = C G_f C$ is also a positive semi-definite trace--class operator. 
\end{proof}

Given a square--summable sequence, $\mbf{f} = (f_j) \subseteq \cJ$ and $f = \sum h_j \otimes f_j \in \cH \otimes \cJ$, let $\Ga _f$ be the Gramian of $\mbf{f}$ and set $P_f := P _{\nbran \Ga _f ^{-\| \cdot \|}} = P _{\nbker \Ga _f} ^\perp$. Then $\Ga _f = P_f \Ga _f P_f$, and we will call $P_f$ the \emph{support projection of $f$}.

\begin{lemma} \label{Gramlinindep}
Let $(f_i) \subseteq \cJ$ be square--summable. Then $(f_i)$ is a linearly independent sequence if and only if $\nbker \Ga _f = \{ 0 \}$, if and only if $\nbker G_f =\{ 0 \}$. 
\end{lemma}
\begin{proof}
Given the Gram matrix, $G_f = \sum _{i,j} \ip{f_i}{f_j}_{\cJ} \ip{h_j}{\cdot}_{\cH} h_i \in \TC (\cH)$, of the square--summable sequence, $\mbf{f} := (f_i)$, we have that 
$$ h = \sum _k c_k h_k \in \nbker G_f, $$ if and only if 
$$ 0  =  \sum _{i,k} \ip{f_i}{f_k} c_k h_i = \sum _i \ip{f_i}{\underbrace{\sum _k c_k f_k}_{=:g \in \cJ}}_\cJ h_i, $$ if and only if 
$$ 0 = \ip{f_i}{\sum _k c_k f_k}_\cJ, $$ for all $i \in \N$, if and only if $g = \sum _k c_k f_k =0$. Since $G_f = C \Ga _f C \geq 0$, it follows that $\nbker \Ga _f = \{ 0 \}$ if and only if $\nbker G_f = \{ 0 \}$. Namely, $h \in \nbker \Ga _f$ if and only if $\ov{h} := C h \in \nbker G_f$. 
\end{proof} 

\begin{lemma} \label{Gramptrace}
Given $f = \sum _i h_i \otimes f_i \in \cH \otimes \cJ$, where $\mbf{f} = (f_j) \subseteq \cJ$ is a square--summable sequence, 
$$ \Ga _f = \mr{id} _\cH \otimes \mr{tr} _\cJ \circ \ip{f}{\cdot}_{\cH \otimes \cJ} f. $$
\end{lemma}
In the lemma statement above, $\mr{id} _\cH : \scr{B} (\cH) \rightarrow \scr{B} (\cH)$ denotes the linear identity map on $\scr{B} (\cH)$ and $\mr{tr}_\cJ$ denotes the semi-finite trace over $\cJ$, so that $\mr{id} _\cH \otimes \mr{tr} _\cJ : \TC (\cH \otimes \cJ) \rightarrow \TC (\cH)$ is the \emph{partial trace}. 
\begin{proof}
Let $(g_i)$ be an orthonormal basis of $\cJ$. Then, 
\ba \mr{id} _\cH \otimes \mr{tr} _\cJ \circ \ip{f}{\cdot}_{\cH \otimes \cJ} f & = & 
\mr{id} _\cH \otimes \mr{tr} _\cJ \circ \sum _{j,k} \ip{h_j \otimes f_j}{\cdot}_{\cH \otimes \cJ} h_k \otimes f_k \\
& = & \mr{id} _\cH \otimes \sum _i \ip{g_i}{ (\cdot ) g_i}_\cJ \circ \sum _{j,k} \ip{h_j}{\cdot}_\cH h_k  \otimes \ip{f_j}{\cdot}_\cJ f_k \\
& = & \sum _{j,k} \ip{h_j}{\cdot}_\cH h_k \sum _i \ip{g_i}{f_k} _\cJ \ip{f_j}{g_i}_\cJ \\
& = & \sum _{j,k} \ip{f_j}{f_k}_\cJ \ip{h_j}{\cdot}_\cH h_k = \Ga _f. \ea 
\end{proof}

\begin{lemma} \label{Gramker}
Let $\mbf{f} = (f_i) \subseteq \cJ$ be a square--summable sequence. 
A vector, $h = \sum c_j h_j \in \cH$, belongs to the kernel of the Gramian, $\Ga _f$, of $f = \sum h_i \otimes f_i$, $h \in \nbker \Ga _f $ if and only if $h \otimes g$ is orthogonal to $f$ for every $g \in \cJ$.
\end{lemma}
\begin{proof}
By Lemma \ref{Gramlinindep}, $C h = \sum \ov{c} _i h_i \in \nbker G_f$, so that $\sum \ov{c_i} f_i =0$. Hence, 
\ba \ip{h \otimes g}{f}_{\cH \otimes \cJ} & = & \sum _{i,j} \ov{c_i} \ip{h_i}{h_j}_\cH \ip{g}{f_j}_\cJ \\
& = & \sum _j \ov{c_j} \ip{g}{f_j} = \ip{g}{\sum _j \ov{c_j} f_j} =0. \ea 

Alternatively, if $h \in \nbker \Ga _f$, 
\ba 0 & = & \ip{h}{\Ga _f h}_\cH = \ip{h}{\mr{id} _\cH \otimes \tr _\cJ \circ (ff^*) h} \\
& = & \tr _\cJ \left( h^* \otimes I_\cJ \,  ff^* \, h \otimes I_\cJ \right). \ea 
Hence,
$$\| h^* \otimes I_\cJ ff^* h \otimes I_\cJ \| _{\TC (\cJ)} =0, $$ which implies that 
$$ 0 = h^* \otimes I_\cJ ff^* h \otimes I_\cJ \geq 0 \in \scr{B} (\cJ), $$ and hence 
$h^* \otimes I_\cJ f =0$ and $h^* \otimes g^* f =0$ for every $g \in \cJ$.

Conversely, if $h \otimes g \perp f$ for every $g \in \cJ$, then $h ^* \otimes I _\cJ : \cH \otimes \cJ \rightarrow \cJ$ obeys $f \in \nbker h^* \otimes I_\cJ$. Hence, by the previous Lemma \ref{Gramptrace}, 
\ba \ip{h}{\Ga _f h}_\cH & = & \ip{h}{\mr{id} _\cH \otimes \tr _\cJ \circ (ff^*) h} \\
& = & \tr _\cJ \left( h^* \otimes I_\cJ ff^* h \otimes I_\cJ \right) =0, \ea as $f \in \nbker h^* \otimes I_\cJ$ by assumption. Hence, since $\Ga _f \geq 0$, it follows that $\Ga _f h =0$ and $h \in \nbker \Ga _f$. 
\end{proof}

\begin{lemma} \label{Gramker2}
Given a square--summable sequence, $\mbf{f}= (f_j ) \subseteq \cJ$ and $f = \sum h_j \otimes f_j \in \cH \otimes \cJ$ with support projection, $P_f$, a linear operator $A \in \scr{B} (\cH, \cH ')$ obeys $A P_f =0$ if and only if $A \otimes I_\cJ f =0$. 
\end{lemma}
\begin{proof}
Given $A \in \scr{B} (\cH, \cH')$, $A P_f =0$ if and only if $A \Ga _f =0$, or equivalently $A \Ga _f A^* =0$. Hence, 
\ba A \Ga _f A^* & = & A (\mr{id} _\cH \otimes \mr{tr} _\cJ \circ ff^*) A^* \\
& = & \mr{id} _{\cH'} \otimes \mr{tr} _{\cJ} \circ (A \otimes I_\cJ f f^* A^* \otimes I_\cJ), \ea and since $\mr{id}_{\cH'} \otimes \mr{tr} _\cJ$ is faithful, it follows that $A \Ga _f A^* =0$ if and only if 
$$ 0 = A \otimes I_\cJ f f^* A^* \otimes I_\cJ, $$ or equivalently, if and only if $A \otimes I_\cJ f =0$.

\end{proof}

\begin{cor} \label{supportproj}
Given any square--summable sequence $\mbf{f} = (f_j) \subseteq \cJ$ and $f = \sum h_i \otimes f_i \in \cH \otimes \cJ$, we have that  $P_f \otimes I_\cJ f = f$, where $P_f = P_{\nbker \Ga _f} ^\perp$ is the support projection of $f$.  
\end{cor} 
\begin{proof}
Given any $x = \sum h_i  \otimes f_i ' \in \cH \otimes \cJ$, 
$$ (I-P_f) \otimes I_\cJ x = \sum (I-P_f) h_i \otimes f_i '. $$ Since each $(I-P_f) h_i \in \nbker \Ga _f$, it follows from the previous lemma that $(I-P_f) h_i  \otimes f_i' \perp f$ for every $i$. Hence $x \perp f$. This shows that $f \perp \nbran (I-P_f) \otimes I_\cJ$ so that $f \in \nbran P_f \otimes I_\cJ$.  
\end{proof}

Given a square--summable sequence, $\mbf{f}=(f_i) \subseteq \cJ$, let $\Ga _f \geq 0$, $\Ga_f \in \TC (\cH)$ be the Gramian of $f = \sum _j h_j \otimes f_j$.  If $(h ' _i)$ is any fixed orthonormal basis of $\nbran P_f \subseteq \cH$, where recall that $P_f$ is the projection onto $\nbker \Ga _f ^\perp$, the support projection of $f$, then observe that 
\ba \Ga _f & = & P_f \Ga _f P_f  \\
& = & \mr{id} _\cH \otimes \tr _\cJ \circ \left( P_f \otimes I_\cJ ff ^* P_f \otimes I_\cJ \right) \\
& = & \mr{id} _\cH \otimes \tr _\cJ \circ \sum _{i,j,k,\ell} h_i ' h_i ^{'*} h_j h_k^* h_{\ell} ' h_\ell ^{'*} \otimes f_j f_k ^* \\
& = & \sum _{i,j,k,\ell} h_i ' h_i ^{'*} h_j h_k^* h_{\ell} ' h_\ell ^{'*} f_k^* f_j  h_i ' h_{\ell} ^{' *} \\
& = & \sum _{i, \ell} \ip{\sum _k \ip{h_\ell '}{h_k}_\cH f_k }{\sum _j \ip{h_i '}{h_j} f_j}_\cJ h_i ' h_\ell ^{'*} \\
& =: & \sum _{i,\ell} \ip{f_\ell '}{f_i '}_\cJ h_i ' h_\ell ^{'*}. \ea 
Hence, restricting $\Ga _f$ to $\cH' := \nbran P_f$, we obtain
$\Ga _f | _{\cH'} : \cH ' \rightarrow \cH'$, obeys $\Ga _f | _{\cH '} = \Ga _{f'}$, where $\Ga _{f'}$ is the Gram matrix of $f = f' := \sum _i h_i ' \otimes f_i ' \in \cH' \otimes \cJ$, with respect to the orthonormal basis, $(h ' _i )$, of $\cH ' = \nbran P_f$, and 
$$ f_i ' = \sum _j \ip{h_i '}{h_j} f_j \in \cJ, $$ form a square--summable and linearly independent sequence in $\cJ$ since $\Ga _{f'} = \Ga _f | _{\cH'} > 0$ is injective. 
In summary, we have proven the following lemma:
\begin{lemma} \label{liconstruct}
Let $\mbf{f} = (f_j) \subseteq \cJ$ be a square--summable sequence, and let $\Ga _f \in \TC (\cH)$, $\Ga _f \geq 0$ be the Gramian of $f = \sum h_i \otimes f_i$ with respect to a fixed orthonormal basis, $(h_j)$, of $\cH$. If $(h' _j )$ is an orthonormal basis of $\cH' := \nbran P_f \subseteq \cH$, then 
$$ \mbf{f} ' := ( f' _j) \subseteq \cJ; \quad \quad f_i ' = \sum _j \ip{h_i '}{h_j} f_j \in \cJ, $$ is a linearly independent and square--summable sequence so that $f = f' = \sum h_i ' \otimes f_i '$.
\end{lemma}

This next theorem relates singularity of a linear pencil evaluated at $\Xi \in \mf{M} ^{1\times d}$ to the (analytic) relation co-dimension, $\Delta _1 (\Xi)$.

\begin{thm} \label{zerocomp}
Let $A \in \scr{B} (\cH) ^{d+1}$ and let $\Xi \in \mf{M} ^{1 \times d}$ have maximal analytic relation co-dimension, $\Delta_1 (\Xi ) =d$. If there exist non-zero $f,g \in \cH \otimes \cJ$ so that $\hat{A} (\Xi ) f = 0 = \hat{A} (\Xi ) ^* g$, then $0 = P_g A^{(k)} P_f$ for all $0 \leq k \leq d$, or equivalently, $0 = \Ga _g A^{(k)} \Ga _f$.
\end{thm}

\begin{proof}
If $\Delta_1 (\Xi) =d$ and $A(\Xi ) f = 0 = A (\Xi ) ^* g$ where $g = \sum h_i \otimes g_i$ and $f = \sum h_i \otimes f_i$ have linearly independent components, \emph{i.e.} $(g_i) \subseteq \cJ$ and $(f_j) \subseteq \cJ$ are square--summable sequences, then let $T_k$, $0 \leq k \leq d$ be as in Lemma \ref{kercomm}. Fix $k$ and set $T:= T_k$ and $A:= A^{(k)}$. Then for any $p,q \in \N$, 
\ba \ip{f_q}{T g_p}_\cJ & = & \ip{f_q}{\sum _{i,j} A _{i,j} \ip{g_i}{g_p}_\cJ f_j}_\cJ \\
& = & \sum _{i,j} A_{i,j} \ip{g_i}{g_p} _\cJ \ip{f_q}{f_j}_\cJ. \ea 
Consider the product
\ba \Ga _g A & = & \sum _{i,j,k,\ell} \ip{g_i}{g_j}_\cJ \ip{h_i}{h_k}_\cH A_{k,\ell} \ip{h_\ell}{\cdot}_\cH h_j \\
& = & \sum _{i,j,\ell} \ip{g_i}{g_j}_\cJ A_{i,\ell} \ip{h_\ell}{\cdot}_\cH h_j, \ea and then the product,
\ba \Ga _g A \Ga _f & = & \sum _{\substack{i,j,\ell \\ m,n}} \ip{g_i}{g_j}_\cJ A_{i,\ell} \ip{h_\ell}{h_n}_\cH \ip{f_m}{f_n}_\cJ \ip{h_m}{\cdot}_\cH h_n \\
& = & \sum _{j,m} \left( \sum _{i, \ell} \ip{g_i}{g_j}_\cJ A_{i,\ell} \ip{f_m}{f_\ell}_\cJ \right) \ip{h_m}{\cdot}_\cH h_j. \ea 
It then follows from direct comparison that 
$$ (\Ga _g A ^{(k)} \Ga _f)_{p,q} = \ip{h_p}{\Ga _g A \Ga _f h_q}_{\cH} = \ip{f_q}{T f_p}_\cJ = T_{p,q} = (T_k)_{p,q}. $$
Since we assume that $\Delta_1 (\Xi ) =d$, this implies that each $T_k =0$ for $1\leq k \leq d$, and hence $\Ga _g A^{(k)} \Ga _f =0$, or equivalently, $P_g A^{(k)} P_f =0$ for $1\leq k \leq d$. Finally, since $P _g A^{(k)} P _f =0$ for $1 \leq k \leq d$,
\ba 0 & = & P_g \otimes I_\cJ A(\Xi) f = P_g \otimes I_\cJ A(\Xi) P_f \otimes I_\cJ f  \\
& = & P_g A^{(0)} P_f \otimes I_\cJ f, \ea 
and we conclude that also $P_g A^{(0)} P_f =0$, and  $0 =\Ga _g A^{(0)} \Ga _f$ by Lemma \ref{linindep}, since the components of $f$ are linearly independent with respect to any fixed orthonormal basis of $\cH ' = \nbran P_f$ by Lemma \ref{liconstruct}. Alternatively, the same conclusion follows by applying Lemma \ref{Gramker2}. Namely, by Lemma \ref{Gramker2}, since $P_g A^{(0)} P_f \otimes I_\cJ f =0$, we have that $P_g A^{(0)} P_f =0$.
\end{proof}

\section{Application to $\scr{O} _d (r)$} \label{sec:application}

\begin{prop} \label{invatzero}
If $H \in \scr{O} _d (r) ^{\ell \times \ell}$ and $\Xi \in \mf{M} ^{1\times d}$ obey $\Delta_1 (\Xi) =d$, $\rownorm{ \Xi } <r$, then $\sigma _p (H(\Xi)) \subseteq \sigma (H(0))$ and $\sigma _p (H(\Xi)^*) \subseteq \sigma (H(0)^*)$. 
\end{prop}
\begin{proof}
We prove the first statement, the second statement follows from a symmetric argument. 

Let $H \sim _0 (A,B,C,D)$ where $A \in \scr{B} (\cH ^{\ell^2})$, $B \in \scr{B} (\C ^\ell, \cH ^{\ell^2})$, $C \in \scr{B} (\cH ^{\ell ^2}, \C ^\ell)$ and $D = H (0) \in \C ^{\ell \times \ell}$. If $\la \notin \sigma (H(0))$, then $D - \la I_\ell$ is invertible and we consider 
$H - \la I_\ell \sim _0 (A,B,C, D - \la I_\ell)$. By Proposition \ref{odrreal} and Proposition \ref{realatY}, we can and do assume that $\| A \| _{\mr{col}} \cdot \| \Xi \| _{\mr{row}} <1$ so that $H(\Xi)$ is well-defined.

Suppose that $\la \in \sigma _p (H (\Xi))$ so that $H(\Xi) - \la I_{\cJ ^\ell}$ has kernel. Choose $0 \neq x \in \cJ ^\ell$ so that $H(\Xi) x = \la x$. Then, by Proposition \ref{Schursing},
$$ f = \Phi ^{-1} x = -L_A (\Xi ) ^{-1} B(\Xi) x \in \nbker L_{A^{(-1)}} (\Xi), $$ where here, $(H - \la I_\ell) ^{-1} \sim _0 (A^{(-1)}, B^{(-1)}, C^{(-1)}, (D- \la I_\ell) ^{-1})$ so that 
\[
    L_{A^{(-1)}} (\Xi) 
    = I_{\cH ^{\ell ^2} \otimes \cJ} - A \otimes \Xi + B(D-\la I_\ell) ^{-1}C \otimes \Xi.
\]

Let $P := P_f \in \scr{B} (\cH ^{\ell^2})$ be the support projection of $f \in \cH ^{\ell ^2} \otimes \cJ$. Compress the realization of $H_\la := H-\la I_\ell$ by $P$ to obtain $H_\la '$, the NC function given by the compressed realization, $(A', B' , C', D -\la I_\ell)$, where 
$\nbran P =: \cH '$, 
$$ A' _j := P A_j | _{\cH'}, \quad B' _j := P B_j, \quad \mbox{and} \quad C' := C | _{\cH '}. $$
Then, since $H' _\la (0) = D -\la I_\ell$ is invertible, 
$$ (H' _\la) ^{-1} \sim _0 (A^{' (-1)}, B^{' (-1)}, C  ^{'(-1)}, (D -\la I_\ell) ^{-1} ), $$ where, 
\ba 
    A _j ^{' (-1)} 
        & = & A' _j - B' _j (D -\la I_\ell ) ^{-1} C' \\
        & = &  P \restr{\left(A_j - B_j (D -\la I_\ell ) ^{-1} C \right)}{\cH'} = P A_j ^{(-1)} | _{\cH' = \nbran P}, \\
    B_j ^{' (-1)} 
        & = & -B_j ' (D - \la I_\ell ) ^{-1} = P B_j ^{(-1)}, \\
    C ^{' (-1)} 
        & = & (D-\la I_\ell ) ^{-1} C' = C ^{(-1)}| _{\cH'}, 
\ea
and $H_\la ^{-1} \sim _0 (A^{(-1)}, B^{(-1)}, C^{(-1)}, (D - \la I_\ell) ^{-1} )$. That is, the realization of $(H' _\la) ^{-1}$ is obtained by compressing the realization of $H_\la ^{-1}$ by $P = P_f$. 

Recall, that by Proposition \ref{realatY}, we know $\colnorm{A} \cdot \rownorm{\Xi}<1$, thus
\[
    \colnorm{A'} \cdot \rownorm{\Xi} 
        \leq \colnorm{A} \cdot \rownorm{\Xi} < 1,
\]
and hence $\Xi$ is in the invertibility domain of the realization $(A',B',C', D)_Y$ and $H' _\la (\Xi)$ is well-defined. 
Moreover $H' _\la (0) = D-\la I_\ell$ is invertible, hence $H' _\la$ is non-zero. 
In fact, Proposition \ref{Schursing} tells us that we have a linear isomorphism, 
\[
    \Phi ' : \nbker L_{A^{'(-1)}} (\Xi) \twoheadrightarrow \nbker H' _\la (\Xi),
\]
$\Phi ' = - (D - \la I_\ell)^{-1} C' \otimes I_\cJ$, with inverse, $(\Phi ') ^{-1} := -L_{A'} (\Xi ) ^{-1} B' \otimes \Xi$. 
Since $P \otimes I_\cJ f = f$, by Corollary \ref{supportproj},
\[
    L_{A^{' (-1)}} (\Xi ) f 
        = P \otimes I_\cJ L_{A^{(-1)}} (\Xi) P \otimes I_\cJ f = 0,
\]
so that $f =\Phi ^{-1} x \in \nbker  L_{A^{' (-1)}} (\Xi )$.
Hence, $\Phi ' f = x' \in \nbker H' _\la (\Xi)$, 
\ba 
    x' 
        & = & \Phi ' f = (D-\la I_\ell) ^{-1} C | _{\nbran P = \cH'} \otimes I\cJ f \\
        & = & (D-\la I_\ell) ^{-1} C \otimes I_\cJ (P \otimes I_\cJ) f \\
        & = & (D-\la I_\ell) ^{-1} C \otimes I_\cJ f = \Phi f = \Phi \Phi ^{-1} (x) = x. 
\ea 
That is, $0 \neq x \in \nbker H' _\la (\Xi)$. 
Since $\mf{M} ^{\ell \times \ell}$ is a finite von Neumann algebra, there exists $0 \neq y \in \nbker (H' _\la (\Xi )) ^*$, and hence, by Proposition \ref{Schursing}, there exists a $0 \neq g = (\Psi ' ) ^{-1} (y) \in \nbker L_{A^{' (-1)}} (\Xi)$. 
If $Q = P_g$ is the support projection of $g$ in $\scr{B} (\cH')$ where $\cH' = \nbran P$, it follows that we can view $Q \leq P$ as an orthogonal projection dominated by $P \in \scr{B} (\cH ^{\ell ^2})$. 
In conclusion, by Theorem \ref{zerocomp}, since $L_{A^{(-1)}}$ is a monic affine-linear pencil,
\[
    0 = Q I_{\cH ^{\ell^2}} P = Q,
\]
contradicting that $Q$ is a non-zero projection since $Q=P_g$ and $g \neq 0$.
\end{proof}

\begin{remark}
By \cite[Lemma 5.8, proof of Corollary 5.9]{AMS-ncentire}, any full matrix, $H \in \scr{O} _d (r) ^{\ell \times \ell}$ is invertible at some matrix-point, $Y \in r \cdot \rball$, and if $Y \in \cdn$, then the set of all $X \in r \cdot \bdn$ at which $H(X)$ is invertible is analytic-Zariski dense in $r \cdot \bdn$. 
Our original strategy to extend the conclusion of the previous proposition to arbitrary full $H \in \scr{O} _d (r) ^{\ell \times \ell}$ that are not necessarily invertible at $0$, was to instead apply similar arguments to the realization of $H$ centred at a matrix-point $Y \in r \cdot \bdn$ at which $H(Y)$ is invertible. This, however, ran into technical difficulties that we have not yet been able to overcome. Instead, given such a matrix-point, $Y$, the following arguments allow us to construct, from the $d-$tuple, $\Xi= (\Xi_1, \cdots, \Xi _d)$, a $dn^2-$tuple, 
$$ \Ups = (\Ups ^{(1)}, \cdots, \Ups ^{(d)}); \quad \quad \Ups ^{(k)} _{i,j} = e_{1,i} \Xi _k e_{j,1} - y^{(k)} _{i,j} I, $$ where 
$$ Y = (Y ^{(1)}, \cdots, Y ^{(d)} ); \quad \quad Y ^{(k)} = (y^{(k)} _{i,j}) _{i,j =1} ^n \in \C ^{n \times n}, $$ and the $e_{1,i}, e_{j,1}$ are the embeddings of the standard $n\times n$ matrix units in $(\mf{M},\tau) *_\C (\C ^{n\times n}, \tr _n)$, so that $\Delta _1 (\Xi ) = d$ if and only if $\Delta _1 (\Ups ) = dn^2$. This will allow us to ``subtract away $Y$" and reduce the problem to the case of $H$ being invertible at $0$. 
\end{remark}

In \cite{CS-hom}, Connes and Shlyakhtenko define the quantity $\Delta(\Xi_1,\cdots,\Xi_d)$ for a self-adjoint set of operators, $\Xi$, that generate $\mf{M}$. Here, by a self-adjoint set, we mean that for every $1 \leq j \leq d$, $\Xi_j^* \in \{\Xi_1,\cdots,\Xi_d\}$. In \cite[Theorem 3.3]{CS-hom}, they establish properties of the invariant $\Delta$. The relationship between $\Delta$ and $L^2-$homology relies on the self-adjointness assumption. However, the general properties of subadditivity and additivity of $\Delta$ along free sets \cite[Theorem 3.3(d)]{CS-hom} and \cite[Theorem 3.3(e)]{CS-hom}, respectively, do not require this assumption of self-adjointness.

\begin{lem} \label{lem:matrix_delta}
Let $\Xi_1,\cdots,\Xi_d \in \mf{M}$ form a self-adjoint set. Assume that $\alg(\Xi_1,\cdots,\Xi_d) \cong \C^{n \times n}$, where the inclusion in $\mf{M}$ is unital and trace-preserving. Then, 
\[
\Delta_1(\Xi_1, \cdots,\Xi_d) = \Delta(\Xi_1,\cdots,\Xi_d) = 1 - \frac{1}{n^2}.
\]
\end{lem}
\begin{proof}
The second equality is \cite[Theorem 3.3(b)]{CS-hom} together with their calculation of the Betti numbers of the matrix algebra. To establish the first equality, however the identification with the first torsion functor for the relations in \cite{CS-hom} is unclear. Instead, we will apply finite-dimensionality and the fact that $\Xi_1,\cdots,\Xi_d$ generate $\C^{n \times n}$ as a unital algebra to prove the first equality.
    Let us write $F = \{\Xi_1,\cdots,\Xi_d\}$. 
    We first note that our assumption implies that $\mf{M} \cong \C^{n \times n} \otimes \mf{M}_0$, where $\mf{M}_0$ is cut-down of $\mf{M}$ by the first matrix unit equipped with the appropriate trace $\tau_0$, so that $\tau = \tr_n \otimes \tau_0$. 
    Therefore, $L^2(\mf{M},\tau) \cong \C^{n \times n} \otimes L^2(\mf{M}_0,\tau_0)$. 
    With this identification, the Hilbert--Schmidt operators can be identified with $\HS(L^2(\mf{M},\tau)) \cong \C^{n \times n} \otimes (\C^{n \times n})^{\opp} \otimes L^2(\mf{M}_0,\tau_0) \overline{\otimes} L^2(\mf{M}_0,\tau_0)$. 
    On this space, the Connes--Shlyakhtenko differential operator is given by $\partial_{F} \colon  \HS(L^2(\mf{M},\tau))^{\oplus d} \to \HS(L^2(\mf{M},\tau))$,
    \[ \partial_{F}(a \otimes b \otimes \xi \otimes \eta \otimes e_j) = ((a \Xi_j) \otimes b - a\otimes (\Xi_j b)) \otimes \xi \otimes \eta.
    \]
    Therefore, the kernel of the differential is of the form $V \otimes L^2(\mf{M}_0,\tau_0) \overline{\otimes} L^2(\mf{M}_0,\tau_0)$, where $V \subseteq \C^{n \times n} \otimes (\C^{n \times n})^{\opp} \otimes \C^d$ is the kernel of the differential corresponding to $\C ^{n \times n}$ itself. In particular, the projection onto the kernel belongs to the algebra $\C^{n \times n} \otimes (\C^{n \times n})^{\opp}$. Let us denote this projection by $p$. If $T \in \ker \partial_{F}$, we can find a sequence of finite ranks $T_m \to T$ in the Hilbert-Schmidt norm. Hence, $p T_m \to T$, but $p T_m$ is in the kernel of the differential operator. Hence, $\Delta_1(F) = \Delta(F)$.
\end{proof}

Assume that $\mf{M} = W^*(\Xi_1,\cdots,\Xi_d)$. Consider the free product $(\mf{N},\psi) = (\mf{M},\tau) *_{\C}(\C^{n \times n}, \tr_n)$. Let $e_{ij}$ be the matrix units of $\C^{n \times n}$ embedded in $\mf{N}$. Let $\mf{N}_0 = e_{11} \mf{N} e_{11}$ and equip $\mf{N}_0$ with the trace $\psi_0 = n \psi|_{\mf{N}_0}$.

\begin{lem}
    There is a *-isomorphism $\Theta \colon \mf{N} \to \C^{n \times n} \otimes \mf{N}_0$, such that $(\tr_n \otimes \psi_0) \circ \Theta = \psi$.
\end{lem}
\begin{proof}
Let $\Theta(a) = \sum_{i,j=1}^n E_{ij} \otimes e_{1i} a e_{j1}$, where $E_{i,j}$ are the standard matrix units of $\C ^{n \times n}$. Clearly, this a unital linear map. Moreover,
\[
\Theta(a)^* = \sum_{i,j=1}^n E_{ji} \otimes e_{1j} a^* e_{i1} = \Theta(a^*),
\]
and hence, $\Theta$ is self-adjoint. Now,
\[
\Theta(a) \Theta(b) = \sum_{i,j,k,\ell = 1}^n E_{ij} E_{k \ell} \otimes e_{1i} a e_{jk} b e_{\ell 1} = \sum_{i,j,\ell=1}^n E_{i\ell} \otimes e_{1i} a e_{jj} b e_{\ell 1} = \Theta(ab).
\]
The last equality follows from summing first the $j$ index and using the fact that the middle term of the product becomes the identity, and we conclude that $\Theta$ is a $*-$homomorphism.

Define $\La \colon \C^{n \times n} \otimes \mf{N}_0 \to \mf{N}$ by $\La \left( \sum_{i,j=1}^n E_{ij} \otimes a_{ij}\right) = \sum_{i,j=1}^n e_{i1} a_{ij} e_{1j}$. Then,
\[
\La(\Theta(a)) = \La\left( \sum_{i,j=1}^n E_{ij} \otimes e_{1i} a e_{j1} \right) = \sum_{i,j=1}^n e_{ii} a e_{jj} = a.
\]
Similarly,
\[
\Theta\left( \La \left( \sum_{i,j=1}^n E_{ij} \otimes a_{ij}\right) \right) = \sum_{i,j=1}^n \Theta(e_{i1} a_{ij} e_{ij}) = \sum_{i,j, k, \ell = 1}^n E_{k\ell} \otimes e_{1k} e_{i1} a_{ij} e_{1j} e_{\ell 1} = \sum_{i,j=1}^n E_{ij} \otimes a_{ij}.
\]
Hence, $\Theta$ is a *-isomorphism with inverse $\Theta ^{-1} = \La$. Moreover,
\[
(\tr_n \otimes \psi_0)(\Theta(a)) = \sum_{i,j = 1}^n \tr_n(E_{ij}) \psi_0(e_{1i} a e_{j1}) = \sum_{i = 1}^n \psi(e_{1i} a e_{i1}) = \sum_{i = 1}^n \psi(e_{ii} a) = \psi(a).
\]
The penultimate equality is due to the fact that $\psi$ is a trace.
\end{proof}
Consider the natural trace-preserving embedding $\iota \colon \mf{M} \to \mf{N}$. We have an induced embedding $\Theta \circ \iota \colon \mf{M} \to \C^{n \times n} \otimes \mf{N}_0$. In particular, if we set 
\be \Omega_k := \Theta(\Xi_k) = \sum_{i,j=1}^n E_{ij} \otimes e_{1i} \Xi_k e_{j1}, \label{Omdef} \ee then, $\Theta \circ \iota$ induces a trace preserving isomorphism from $\mf{M}$ to $W^*(\Omega_1,\cdots,\Omega_d)$. 

For the purpose of the next proof we require a rather mundane claim about the expectation $\tr_n \otimes I \colon \C^{n \times n} \otimes \mathscr{B}(\cH) \to \mathscr{B}(\cH)$. Let us fix a basis for $\mf{sl} _n(\C)$ given by $E_{i,j}$ for $i \neq j$ and $E_{1,1} - E_{i,i}$ for $2 \leq i \leq n$. Then, every $T \in \C^{n \times n} \otimes \mathscr{B}(\cH)$ can be written as $T = \frac{1}{n}I \otimes T_0 + \sum_{1 \leq i \neq j \leq n} E_{ij} \otimes T_{ij} + \sum_{i=2}^n (E_{11} - E_{ii}) \otimes T_i$. Clearly, $(\tr_n \otimes I) = T_0$. Hence, $T \in \ker \, \tr_n \otimes T_0$ if and only if $T_0 = 0$. Now write
\[
E_{11} - E_{ii} = [E_{1i}, E_{i1}] \text{ and } E_{ij} = [E_{i1},E_{1j}], \text{ for } i \neq j.
\]
Therefore, if we assume that $T \in \nbker \tr_n \otimes I$, then
\ba
T & = & \sum_{1 \leq i \neq j \leq n} [E_{i1}, E_{1j}] \otimes T_{ij} + \sum_{i=2}^n [E_{1i}, E_{i1}] \otimes T_i \\
& = & \sum_{1 \leq i \neq j \leq n} [E_{i1} \otimes I, E_{1j} \otimes T_{ij}] + \sum_{i=2}^n [E_{1i} \otimes I, E_{i1}\otimes T_i]. \ea
In particular, $T$ is finite-rank (respectively, trace-class) if and only if each of its coordinates is. Hence, every element in the kernel of the expectation that belong to the ideal of finite-rank or trace-class operators can be written as commutators with the set of ``matrices'' $E = \{E_{1i}\}_{i=2}^n \cup \{E_{i1}\}_{i=1}^n$. The latter set can easily be seen to generate $\C^{n \times n}$ as an algebra and it is self-adjoint in the sense of Connes and Shlyakhtenko. Therefore, if we consider the algebra $\C^{n \times n} \otimes \mf{N}_0$ as above, then by Lemma \ref{lem:matrix_delta}, we know that $\Delta(E) = \Delta_1(E) = 1 - \frac{1}{n^2}$ and $|E| = 2n - 1$.

\begin{prop} \label{thm:delta_equality}
    Using the notations above, let $F = \{e_{1i} \Xi_k e_{j1} \mid 1 \leq i, j \leq n,\, 1 \leq k \leq d\}$. Then, $\Delta(\Xi_1,\cdots,\Xi_d) = d$ if and only if $\Delta(F) = d n^2$. The same assertions hold for $\Delta _1$.
\end{prop}
\begin{proof}
Let $\theta \colon \C^{n \times n} \otimes \mf{N}_0 \to \mf{N}_0$ be the conditional expectation given by $\theta = \tr_n \otimes \operatorname{id}_{\mf{N}_0}$. Set $\cK = L^2(\C^{n \times n} \otimes \mf{N}_0,\tr_n \otimes \psi_0)$ and $\cH = L^2(\mf{N}_0,\psi_0)$. Let $\eta$ and $\xi$ be the corresponding cyclic vectors. The map $V (a\xi) = a \eta$ for $a \in \mf{N}_0$ extends to an isometric embedding of $\cH$ in $\cK$. Moreover, it is clear that $(\tr_n \otimes \psi_0) \circ \theta = \tr_n \otimes \psi_0$. By \cite[Proposition II.6.10.7]{Blackadar-op_alg}, we have that $\theta$ is the restriction to $\C^{n \times n} \otimes \mf{N}_0$ of the projection from $\cK$ to $\cH$. We note that $\cK \cong L^2(\C^{n \times n}) \otimes \cH \cong \C^{n \times n} \otimes \C^{n \times n} \otimes \cH$. Therefore, we can identify $\mathscr{B}(\cK) \cong \C^{n \times n} \otimes \C^{n \times n} \otimes B(\cH)$. The $\Omega_k$ act on $\cK$ by left multiplication. Hence, we can write this operator as $L_{\Omega_k} = I_n \otimes \Omega_k$.

Suppose that $\Delta(F) = d n^2$ and assume that there are $T_1,\cdots,T_d \in \mathscr{B}(\cK)$ of finite rank (equivalently, trace-class), such that $\sum_{k=1}^d [T_k, L_{\Omega_k}] = 0$. If we write $T_k = \sum_{i,j = 1}^d E_{ij} \otimes T_{ij}^{(k)}$, where each $T_{ij}$ is finite-rank (respectively, trace-class), then the above commutator can be rewritten as
\[
\sum_{k=1}^d \sum_{i,j=1}^n E_{ij} \otimes [T_{ij}^{(k)}, \Omega_k] = 0.
\]
In particular, if we fix $1 \leq p, q \leq m$, then $\sum_{k=1}^d [T_{pq}^{(k)}, \Omega_k]$. Compressing by $V$, we obtain that $\sum_{i,j = 1}^n \sum_{k=1}^m [(T_{pq}^{(k)})_{ij}, e_{1i} \Xi_k e_{j1}] = 0$. Therefore, $T_{pq}^{(1)} = \cdots = T_{pq}^{(d)} = 0$ and $\Delta(\Omega_1,\cdots, \Omega_d) = d$. By the isomorphism above $\Delta(\Xi_1,\cdots,\Xi_d) = \Delta(\Omega_1,\cdots,\Omega_d)$ (similarly for $\Delta_1$). 

Conversely, assume that $\Delta(\Xi_1,\cdots,\Xi_d) = d$, and furthermore, suppose that there are finite-rank or trace-class operators, $t_{ij}^{(k)}$, on $\cH$, such that $\sum_{i,j = 1}^n \sum_{k=1}^d [t_{ij}^{(k)}, e_{1i} \Xi_k e_{j1}] = 0$. Let $T_k = \sum_{i,j = 1}^n I \otimes E_{ij} \otimes t_{ij}^{(k)}$ and $R = \sum_{k=1}^d [T_k, \Omega_k]$. Clearly, $R$ is a finite-rank operator on $\cK$. As in the discussion preceding the proposition, every finite-rank (respectively, trace-class) operator in the kernel of the expectation can be written as a commutator with the set $E$. Let $F' = \{\Omega_1,\cdots,\Omega_d\} \cup E$. Recall that $\C^{n \times n} \otimes \mf{N}_0 \cong \mf{N} = \mf{M} *_{\C} \C^{n \times n}$. The isomorphism, $\Theta$ maps any constant $A = \sum_{i,j=1}^n a_{ij} e_{ij}$ to the matrix $\sum_{i,j=1}^n a_{ij} E_{ij} \otimes e_{11}$. Therefore, the $\Omega_k = \Theta(\Xi_k)$ are free from the scalar matrices and hence by \cite[Theorem 3.3(e)]{CS-hom}, 
\[
\Delta(F') = \Delta(\Omega_1,\cdots,\Omega_k) + \Delta(E) = d + 1 - \frac{1}{n^2}.
\]
On the other hand, let us denote the relation module of $F'$ by $\cR(F')$. Namely,
\[
\cR(F') = \left\{  (S_1,\cdots, S_{d + 2n -1}) \in \FR(\cK) \left| \  \sum_{k=1}^d [S_k, \Omega_k] + \sum_{\ell = 1}^n [S_{d+\ell}, E_{i1}] + \sum_{\ell=2}^n [S_{d + n + i - 1}, E_{1i}] = 0 \right. \right\}^{-HS}.
\]
Therefore,
\[
d + 1 - \frac{1}{n^2} = \Delta(F') = d + 2n - 1 - \dim_{(\C^{n \times n} \otimes \mf{N}_0) \bar{\otimes} (\C^{n \times n} \otimes \mf{N}_0)^{\opp}} \cR(F').
\]
Hence, $\dim_{(\C^{n \times n} \otimes \mf{N}_0) \bar{\otimes} (\C^{n \times n} \otimes \mf{N}_0)^{\opp} } \cR(F') = 2n - 2 + \frac{1}{n^2}$. On the other hand,
\[
\dim_{(\C^{n \times n} \otimes \mf{N}_0) \bar{\otimes} (\C^{n \times n} \otimes \mf{N}_0)^{\opp}} \cR(E) = 2 n  - 2 + \frac{1}{n^2}.
\]
Moreover, there is a natural embedding of the latter left $\mf{M} \bar{\otimes} \mf{M}^{\opp}$-module into the former, by sending $(S_1,\cdots,S_{2n-1})$ to $(0,\cdots,0,S_1,\cdots,S_{2n-1})$. The above calculations show that the Murray--von Neumann dimensions of these two modules are the same. As seen in \cite{Luck-hilbert_modules}, since the trace on $\C^{n \times n} \otimes \mf{N}_0 \bar{\otimes} (\C^{n \times n} \otimes N_0)^{\opp}$ is faithful, the dimension of any proper sub-module is strictly less that the dimension of the ambient sub-module. Hence, the embedding is an isomorphism and we deduce that $T_k = 0$, for all $1 \leq k \leq d$. All of the above arguments hold if we replace finite-rank with trace-class operators so that our claim holds for both $\Delta$ and $\Delta _1$.
\end{proof}

The following theorem is the main result of this section, it generalizes \cite[Theorem 1.1, (v) $\Rightarrow$ (iii), (ii)]{freefield}.

\begin{thm} \label{main1}
Let $H \in \scr{O} _d (r) ^{\ell \times \ell}$ be full, and let $\Xi \in \mf{M} ^{1\times d}$ obey $\Delta _1 (\Xi ) =d$ and $\| \Xi \| _{\mr{row}} <r$. Then $H(\Xi)$ is non-singular, \emph{i.e.} it has no kernel or co-kernel.
\end{thm}
\begin{proof}
By \cite[Lemma 5.8, Corollary 5.9]{AMS-ncentire}, since $H$ is full, there exists an $n \in \N$ so that the set of all $Y \in r \cdot \bdn$ obeying that $H(Y)$ is invertible is analytic-Zariski dense. By Proposition \ref{odrreal} and Proposition \ref{realatY}, we can choose $Y \in r \cdot \bdn$ with sufficiently small row-norm so that $H(Y)$ is invertible and $H \sim _Y (A,B,C,D)$ with 
$$ \| A \| _{\mr{col}} \cdot ( \| \Xi \| _{\mr{row}} + \| Y \| _{\mr{row}} ) < 1. $$ Here, recall from the construction of Proposition \ref{realatY}, that if we assume each matrix-entry $H _{i,j} \sim _0 (a^{(i,j)}, b^{(i,j)}, c^{(i,j)}) \in \scr{B} (\cH) ^{d} \times \cH \times \cH$, then each $A_j : \C ^{n \times n} \rightarrow \scr{B} (\cH ^{n \ell^2})$, $B _j : \C ^{n \times n} \rightarrow \cH ^{n \ell ^2}$, $C : \cH ^{n \ell ^2} \rightarrow \C ^{n \ell}$ and $D \in \C ^{n\ell \times n \ell}$. 

Suppose that $H (\Xi)$ has non-trivial kernel and (hence) co-kernel. 
Then, for any $x \in \nbker H(\Xi) \subseteq \cJ ^\ell$, and any $v \in \C ^n$, 
\[
f(v) :=  L_A (I_n \otimes \Xi - Y \otimes I_\cJ) ^{-1} B (I_n \otimes \Xi - Y \otimes I_{\cJ}) v \otimes x,
\]
belongs to the kernel of $L _{A^{(-1)}} (I_n \otimes \Xi - Y \otimes I_\cJ )$, where $H^{-1} \sim _Y (A^{(-1)}, B^{(-1)}, C^{(-1)}, D^{-1})$.
Take $P_{f(v)}\in \scr{B}(\cH^{n \ell^2})$ to be the support projection of $f(v) \in \cH ^{n\ell ^2} \otimes \cJ$ and define 
\[
    P := \bigvee _{\substack{v \in \C ^n \\ x \in \nbker H (\Xi)}} P _{f(v)}.
\]
Compress the realization $(A,B,C,D)_Y$ by $P$ to obtain $H' \sim _Y (A', B', C', D)$, where $A' _j : \C^{n\times n} \rightarrow \scr{B} (\cH')$, $\cH ' := \nbran P$ is defined by 
\[
    A' _j (G) := P A_j (G) | _{\cH' = \nbran P},
\]
$B' _j (G) := P B_j$, $C' := C | _{\cH'}$ and $D = H(Y)$. 
Then, firstly since 
\[
    \colnorm{A'}\leq \colnorm{A},
\]
we still have that $\colnorm{A'} \cdot \rownorm{I_n \otimes \Xi - Y \otimes I_\cJ} <1$, so that $H' (I_n \otimes \Xi)$ is well-defined. 
Moreover, $H' (Y) = D = H(Y)$ is invertible, so that $(H') ^{-1} \sim _Y (A^{'(-1)}, B^{'(-1)}, C^{'(-1)}, D ^{-1})$, where 
\[
    A^{'(-1)} _j (G) 
        = A' _j (G) - B' _j (G) D^{-1} C' 
        = P A^{(-1)} (G)| _{\cH' = \nbran P}.
\]
Hence, since $P \otimes I_\cJ f(v) = f(v)$ for every $v \in \C^n$, we have that 
\[
    \nbker L_{A^{'(-1)}} (I_n \otimes \Xi - Y \otimes I_\cJ ) 
        \supseteq \nbker L_{A^{(-1)}} (I_n \otimes \Xi - Y \otimes I_\cJ ) 
        \supseteq \bigvee _{v\in \C^n} f(v).
\]
Indeed, by Proposition \ref{Schursing}, since $\Phi ^{-1} : \nbker H (I_n \otimes \Xi) \twoheadrightarrow \nbker L_{A^{(-1)}} (I_n \otimes \Xi - Y \otimes I_\cJ )$ is a linear isomorphism defined by 
$$ \Phi ^{-1} =  L_A (I_n \otimes \Xi - Y \otimes I_\cJ) ^{-1} B (I_n \otimes \Xi - Y \otimes I_{\cJ ^\ell}) | _{\nbker H (I_n \otimes \Xi)}, $$ with inverse
$$ \Phi = -D ^{-1} C \otimes I_\cJ | _{\nbker L_{A ^{(-1)}} (I_n \otimes \Xi - Y \otimes I_\cJ)}, $$
we have that $f(v) = \Phi ^{-1} v\otimes x$. Moreover, if we consider the corresponding linear isomorphism $\Phi ' : \nbker L_{A ^{'(-1)}} (I_n \otimes \Xi - Y \otimes I_\cJ) \twoheadrightarrow \nbker H' (I_n \otimes \Xi)$, we obtain that \ba \Phi ' & = & - D^{-1} C' \otimes I_\cJ | _{\nbker L_{A ^{'(-1)}} (I_n \otimes \Xi - Y \otimes I_\cJ)} \\
& = & -D^{-1} CP \otimes I_\cJ |_{\nbker L_{A ^{'(-1)}} (I_n \otimes \Xi - Y \otimes I_\cJ)} \\
& = & - D^{-1} C \otimes I_\cJ |_{\nbker L_{A ^{'(-1)}} (I_n \otimes \Xi - Y \otimes I_\cJ)}. \ea

Since,
\ba L_{A ^{'(-1)}} (I_n \otimes \Xi - Y \otimes I_\cJ) f(v) & = & P L_{A ^{(-1)}} (I_n \otimes \Xi - Y \otimes I_\cJ) P \otimes I_\cJ f(v)  \\
& = & P \otimes I_\cJ L_{A ^{(-1)}} (I_n \otimes \Xi - Y \otimes I_\cJ) f(v) =0, \ea 
we have that each $f(v) = \Phi ^{-1} x \otimes v \in \nbker  L_{A ^{'(-1)}} (I_n \otimes \Xi - Y \otimes I_\cJ)$ and hence 
$$ \Phi ' f(v) = \Phi ' P \otimes I_\cJ f(v) = \Phi f(v) = \Phi \circ \Phi ^{-1} v \otimes x = v \otimes x \in \nbker H' (I_n \otimes \Xi). $$ 
That is, 
$$ \nbker  L_{A ^{'(-1)}} (I_n \otimes \Xi - Y \otimes I_\cJ) \supseteq \nbker  L_{A ^{(-1)}} (I_n \otimes \Xi - Y \otimes I_\cJ), \quad \mbox{and} \quad \nbker H ' (I_n \otimes \Xi ) \supseteq \nbker H (I_n \otimes \Xi), $$ are both non-trivial.

Since $\mf{M} ^{n \times n}$ is a finite von Neumann algebra, $\nbker H' (I_n \otimes \Xi ) ^* \neq \{ 0 \}$ and again, by Proposition \ref{Schursing}, $$\nbker L_{A^{'(-1)}} (I_n \otimes \Xi - Y \otimes I_\cJ ) ^* = (\Psi ' ) ^{-1} \nbker H' (I_n \otimes \Xi ) ^* \neq \{ 0 \}, $$ is non-trivial, where $\Psi ' : \nbker L_{A^{'(-1)}} (I_n \otimes \Xi - Y \otimes I_\cJ ) ^* \twoheadrightarrow \nbker H' (I_n \otimes \Xi ) ^*$ is the linear isomorphism between the co-kernels of Proposition \ref{Schursing}. Let $Q$ be the support projection of the co-kernel of $L_{A^{'(-1)}} (I_n \otimes \Xi - Y \otimes I_\cJ )$, \emph{i.e.} the span of all the support projections of each vector in the co-kernel. Then $Q \neq 0$ and $Q \in \scr{B} (\cH')$, so that we can view $Q \in \scr{B} (\cH ^{n \ell ^2})$ with $Q \leq P$ since $\cH' = \nbran P$. Then, we obtain that 
$$ 0 = Q L_{A^{'(-1)}} (-Y) P = Q L_{A^{(-1)}} (-Y) P, $$ and hence, since $Q \leq P$,
$$ 0  =  Q L_{A^{(-1)}} (-Y) Q = Q L_A (-Y) Q -QB(Y) D^{-1} C Q. $$
Equivalently, 
$$ L_{QA|_{\nbran Q}} (-Y) = Q B(Y) D^{-1} C | _{\nbran Q}. $$ 
However, $0_n$, is in the $Y-$invertibility domain of the realization $(QA|_{\nbran Q}, QB, C | _{\nbran Q}, D)$, since $\| Q A | _{\nbran Q} \| _{\mr{col}} \| Y \| _{\mr{row}} \leq \| A \| _{\mr{col}} \cdot \| Y \| _{\mr{row}} < 1$, which is to say that $L_{QA|_{\nbran Q}} (-Y)$ is invertible as a linear operator on $\nbran Q$, while $Q B(Y) D^{-1} C | _{\nbran Q}$ has rank at most $n$, since $C : \cH ^{n \ell ^2} \rightarrow \C ^n$, and $B_j (Y) : \C^n \rightarrow \cH ^{n \ell ^2}$. It follows that $\nbdim \nbran Q \leq n$. Also note that, by similar arguments to above, using Proposition \ref{Schursing},
$$ \nbker L_{Q A ^{(-1)} | _{\nbran Q}} ( I_n \otimes \Xi - Y \otimes I_\cJ) ^* \supseteq \nbker L_{P A ^{(-1)} | _{\nbran P}} ( I_n \otimes \Xi - Y \otimes I_\cJ) ^* \neq \{ 0 \}, $$ and so we obtain that $L_{Q A ^{(-1)} | _{\nbran Q}} ( I_n \otimes \Xi - Y \otimes I_\cJ)$ has non-trivial co-kernel and hence kernel by Proposition \ref{Schursing} again. Indeed, let $(A'', B'', C'', D)_Y$ be the realization obtained by compressing the realization $(A',B',C',D)_Y$ by $Q$ and let $H'' \sim _Y (A'', B'', C'', D)$ denote the corresponding quantized transfer function. Then as before, since $D$ is invertible, the realization for $(H'') ^{-1} \sim _Y (A^{'' (-1)}, B^{''(-1)}, C ^{'' (-1)}, D^{-1})$ is also obtained by compressing the realization of $(H') ^{-1}$ by $Q$.
Hence, for any $g' \in \nbker  L_{A^{'(-1)}} ( I_n \otimes \Xi - Y \otimes I_\cJ) ^*$,
$$ L_{A^{''(-1)}} ( I_n \otimes \Xi - Y \otimes I_\cJ) ^* g'  = Q \otimes I_\cJ L_{A^{'(-1)}} ( I_n \otimes \Xi - Y \otimes I_\cJ) ^* \underbrace{Q \otimes I_\cJ g'}_{=g'} =0. $$ Proposition \ref{Schursing} then implies that $\Psi '' : \nbker L_{A^{''(-1)}} ( I_n \otimes \Xi - Y \otimes I_\cJ) ^* \twoheadrightarrow \nbker H'' (I_n \otimes \Xi ) ^*$, given by 
$$ \Psi '' = - D^{-*} \otimes I_\cJ B '' (I_n \otimes \Xi - Y \otimes I_\cJ) ^*  = - D^{-*} \otimes I_\cJ B '' (I_n \otimes \Xi - Y \otimes I_\cJ) ^* Q \otimes I_\cJ | _{\nbker L_{A^{''(-1)}} ( I_n \otimes \Xi - Y \otimes I_\cJ) ^*}, $$ is a linear isomorphism, so that 
$$ \Psi '' g ' =  \Psi ' Q \otimes I_\cJ g' = \Psi ' g '\in \nbker H'' (I_n \otimes \Xi) ^* \cap H' (I_n \otimes \Xi ) ^*. $$ Again, since $\mf{M} ^{n \times n}$ is a finite von Neumann algebra, $H'' (I_n \otimes \Xi)$ also has non-trivial kernel, and if $0 \neq x'' \in \nbker H'' (I_n \otimes \Xi)$, then $0 \neq f'' = (\Psi '' ) ^{-1} x'' \in \nbker L_{A^{''(-1)}} ( I_n \otimes \Xi - Y \otimes I_\cJ)$.

Now, set $\mf{Z} := (\mf{Z}^{(1)}, \cdots, \mf{Z}^{(d)})$, with $\mf{Z}^{(k)} := ( \mf{z}^{(k)} _{i,j})$, $1 \leq i,j \leq n$ and consider $L_{Q A^{(-1)}| _{\nbran Q}}$ as a monic, affine-linear pencil in the $dn^2$ formal variables $(\mf{z} ^{(k)} _{i,j})$. That is, if 
$$ Z ^{(k)} := \sum _{i,j =1} ^n \underbrace{E_{i,j}}_{\in \C ^{n \times n}} \otimes \underbrace{Z^{(k)} _{i,j}}_{\in \C ^{m \times m}} \in \C ^{mn \times mn}; \quad  \quad 1 \leq k \leq d, $$ then we define
\ba L_{\hat{A}} ( ( Z^{(k)} _{i,j})) & := &  L_{Q A^{(-1)} | _{\nbran Q}} ( Z^{(1)}, \cdots, Z^{(d)} ) \\
& = & I_{\nbran Q} \otimes I_m - \sum _{k=1} ^d \left( Q A^{(-1)} _k |_{\nbran Q} \otimes \mr{id} _m \right) \circ \left( \sum _{i,j=1} ^n E_{i,j} \otimes Z^{(k)} _{i,j} \right) \\
& = & I_{\nbran Q} \otimes I_m - \sum _{\substack{1\leq k \leq d \\ 1 \leq i,j \leq n}} Q A^{(-1)} _k (E_{i,j})|_{\nbran Q} \otimes Z^{(k)} _{i,j}.  \ea 
That is, $\hat{A} :=  \mr{col}(\hat{A} ^{(k)} _{i,j}) _{\substack{1\leq k \leq d \\ 1 \leq i,j \leq n}}$ is the $dn^2$ column-tuple with components, $\hat{A} ^{(k)} _{i,j} := Q A^{(-1)} _k (E_{i,j})|_{\nbran Q}$ in $\scr{B} (\nbran Q)$, and we identify the $d-$tuple $Z = (Z^{(1)}, \cdots, Z^{(d)}) \in \C ^{(mn \times mn) \cdot d}$ with the $dn^2-$tuple $(Z^{(k)} _{i,j}) \in \C ^{(m \times m) \cdot dn^2}$. 

Hence $L_{\hat{A}} (0) = I _{\nbran Q}$ is invertible. We now apply the previous results of this section. Namely, by Proposition \ref{thm:delta_equality}, the $dn^2-$tuple,
$$ (F ^{(k)} _{i,j} ) = (e_{1,i} \Xi _k e_{j,1}) _{\substack{1 \leq i,j \leq n, \\ 1 \leq k \leq d}}; \quad \quad F^{(k)} _{i,j} \in \mf{N} _0 = e_{11} \mf{N} e_{11}, \quad \mf{N} = (\mf{M}, \tau) * _\C (\C ^{n \times n}, \mr{tr} _n),$$ has maximal $\Delta _1 (F) = dn^2$. Since subtracting constant multiples of the identity does not affect $\Delta _1$, the $dn^2-$tuple, 
$$ \Ups = (\Ups ^{(1)}, \cdots, \Ups ^{(d)} ) := \left( e_{1,i} \Xi _k e_{j,1} - y^{(k)} _{i,j} 1_{\mf{N}_0}\right) _{\substack{1 \leq i,j \leq n, \\ 1 \leq k \leq d}} \in \mf{N} _0 ^{1 \times dn^2}, $$ also has maximal $\Delta _1 (\Ups ) = dn^2$, where we write
$$ Y = ( Y^{(1)}, \cdots, Y ^{(k)} ) \in r \cdot \bdn; \quad \quad Y^{(k)} = \sum _{i,j} ^n E_{i,j} \, y^{(k)} _{i,j} \in \C ^{n \times n}.$$  Moreover, $L_{\hat{A}} (\fz)$ is full as an affine linear matrix over $\C \langle \fz _1, \cdots, \fz _{dn^2} \rangle$, since it is monic and finite-rank. (Since the embeddings $\scr{O} _d (r) \hookrightarrow \fps$  and $\fp \hookrightarrow \fps$ are honest, it further follows that $L_{\hat{A}} (\fz)$ is full over $\scr{O} _{dn^2} (r)$ for any $r \in (0, +\infty]$.) Hence,
$$ L_{\hat{A}} \left( (e_{1,i} \Xi _k e_{j,1} - y ^{(k)} _{i,j} I_{\cK _0} )  \right) = L_{\hat{A}} ( F-(y ^{(k)} _{i,j} I_{\cK _0}) ) = L_{\hat{A}} ( \Ups ), $$ where $\cK _0 := L^2 (\mf{N} _0, \psi _0)$ is the tracial GNS space of $\mf{N} _0$, has no kernel or co-kernel by Proposition \ref{invatzero}, since it is a full matrix over $\scr{O} _{dn^2} (r)$ for any $r>0$. (Alternatively, this follows from \cite[Theorem 1.1 (ii)]{freefield}, since this is also a full affine-linear matrix over $\C \langle \fz _1, \cdots, \fz _{dn^2} \rangle$ and we assume that the $d-$tuple $\Xi$ has maximal $\Delta _1$. Since $\Delta \geq \Delta _1$, it must also have maximal $\Delta (\Xi ) =d$ and hence $\Delta (\Ups )=dn^2$ as well.) However,
\ba L_{\hat{A}} (\Ups ) & = &  I _{\nbran Q} \otimes I_{\cK _0} - \sum _{k=1} ^{d} \sum _{i,j =1} ^n Q A^{(-1)} _k (E_{i,j}) | _{\nbran Q} \otimes \left(F^{(k)} _{i,j} -   y^{(k)} _{i,j}  I_{\cK _0} \right) \\ 
&= & I _{\nbran Q} \otimes I_{\cK _0} - \sum _{k=1} ^d \left( Q A^{(-1)} _k (\cdot) | _{\nbran Q} \otimes \mr{id} _{\cK_0} \right) \circ \left( \sum _{i,j =1} ^n E_{i,j} \otimes F^{(k)} _{i,j} - Y ^{(k)} \otimes I_{\cK _0} \right) \\
& = & L _{Q A^{(-1)}| _{\nbran Q}} \left( \Om - Y \otimes I_{\cK _0} \right), \ea where $\Om = (\Om _1, \cdots, \Om _d)$, $\Om _k \in \C ^{n \times n} \otimes \mf{N} _0$ is as defined in Equation (\ref{Omdef}), then has no kernel or co-kernel. 
Since there is a trace-preserving $*-$isomorphism $\Theta \circ \iota : \mf{M} \rightarrow W^* (\Om)$, obeying $\Theta \circ \iota (\Xi _k) = \Om _k$, it follows that $L _{Q A^{(-1)}| _{\nbran Q}} ( I_n \otimes \Xi - Y \otimes I_\cJ )$ has no kernel or co-kernel, contradicting our earlier assertion.
\end{proof}

We now wish to consider evaluations of elements $f \in \scr{M} _d (r)$ at tuples $\Xi \in \mf{M} ^{1\times d}$ or even tuples in $\wt{\mf{M}}$. By \cite[Section 7.4]{Cohn}, elements $f \in \scr{M} _d (r)$ can be identified with certain equivalence clases of `displays', $f \cong \bsm A & B \\ C & D \esm$, where $A \in \scr{O} _d (r) ^{\ell \times \ell}$ is full, $B \in \scr{O} _d (r) ^\ell$, $C \in \scr{O} _d (r) ^{1 \times \ell}$ and $D \in \scr{O} _d (r)$, $\ell \in \N$, and where $f$ is the \emph{transfer function} of this display, 
$$ f (\fz ) = D (\fz) - C (\fz) A(\fz) ^{-1} B (\fz) \in \scr{M} _d (r),$$
that is, $f$ is the Schur complement of the display $M (\fz) = \bsm A(\fz ) & B (\fz) \\ C (\fz) & D (\fz) \esm$ by the invertible upper left block, $A (\fz)$, $f (\fz) = M (\fz) / A (\fz)$.
We will typically write $f \cong \bsm A & B \\ C & D \esm$, with the symbol, $\cong$, to avoid confusion with an FM affine-linear representation (\emph{i.e.} realization) of $f$. By \cite[Corollary 5.9 and Remark 5.10]{AMS-ncentire}, two such displays,
$$ f (\fz ) = D (\fz) -  C(\fz) A (\fz ) ^{-1} B (\fz) = D' (\fz)  - C' (\fz) A' (\fz) ^{-1} B' (\fz ) =: f' (\fz), $$
are equal as elements of $\scr{M} _d (r)$, if and only if their evaluations in the NC universe agree, wherever they are both defined. We will write $\bsm A & B \\ C & D \esm \cong \bsm A' & B' \\ C' & D ' \esm$ and $f \cong \bsm A & B \\ C & D \esm$, if these two displays represent the same element, $f \in \scr{M} _d (r)$. However, given two such displays that represent the same element, it is not obvious, a priori, that the evaluations of these displays need to agree in any stably finite algebra for which the evaluation of both displays is defined. Namely, let $\scr{A}$ be a unital, stably-finite algebra over $\C$. Given $f \in \scr{M} _d (r)$ represented by a display $\bsm A & B \\ C & D \esm$, as above, and given a $d-$tuple $a = (a_1, \cdots, a_d ) \in \scr{A} ^{1 \times d}$, we will say that $a$ belongs to the $\scr{A}-$domain, $\scr{D} _{\scr{A}} (A)$, of the display, if (i) the evaluation $\mr{ev} _a : \fpd \rightarrow \scr{A}$ defined by $\fz _j \mapsto a_j$ extends to a unital homomorphism, $\mr{ev} _a : \scr{O} _d (r) \rightarrow \scr{A}$, and if (ii) $A (a) := \mr{ev} _a (A)$ is invertible over $\scr{A}$.
If $a \in \scr{D} _\scr{A} (A)$, we then define
$$ f (a) := D (a) - C(a) A(a) ^{-1} B(a) \in \scr{A}, $$
where we write $D(a) := \mr{ev} _a (D)$, et cetera. 

This next result is a version of \cite[Theorem 4.4]{freefield} and \cite[Theorem 6.1]{AMS-ncentire}. The former result on evaluation of NC rational functions in stably finite algebras is part of a result originally due to Cohn \cite[Theorem 7.3.2]{Cohn2}.

\begin{thm} \label{evalthm}
Let $\bsm A & B \\ C & D \esm$ and $\bsm A' & B' \\ C ' & D ' \esm$ be displays that represent the same element, $f \in \scr{M} _d (r)$, and let $\scr{A}$ be a unital, stably-finite algebra over $\C$. Then for any $a = (a_1, \cdots, a_d ) \in \scr{D} _\scr{A} (A) \cap \scr{D} _\scr{A} (A')$,
$$ D(a) - C(a) A(a) ^{-1} B(a) = D' (a) -  C' (a) A' (a) ^{-1} B' (a). $$ That is, $f$ is well-defined on 
$$ \mr{Dom} _\scr{A} (f) :=  \bigcup \left\{ \scr{D} _\scr{A} (A) \left| \  f \cong \bsm A & B \\ C &D \esm; \ A \in \scr{O}_d (r) ^{\ell \times \ell}, \ \rho (A) = \ell, \ B \in \scr{O} _d (r) ^{\ell}, C \in \scr{O} _d (r) ^{1 \times \ell}, \ D \in \scr{O} _d (r)  \right. \right\}. $$ 
\end{thm}
\begin{proof}
Suppose that $f \in \scr{M} _d (r)$ and $f \cong \bsm A & B\\ C & D \esm \cong \bsm A' & B' \\ C ' & D' \esm$, where $\bsm A & B \\ C & D \esm$ is of size $\ell$ and $\bsm A' & B'  \\ C' & D' \esm$ is of size $\ell '$. Further assume that there exists $a = (a_1, \cdots, a_d ) \in \scr{D} _\scr{A} (A) \cap \scr{D} _{\scr{A}} (A')$ so that 
$$ D(a) - C(a) A(a) ^{-1} B(a) \neq D' (a) - C' (a) A' (a) ^{-1} B' (a). $$ Consider the block matrix, 
$$ M := \bpm A& 0 & B  \\ 0 & A' & -B' \\ C & C' & D -D' \epm  \in \scr{O} _d (r) ^{(\ell + \ell ' +1) \times (\ell + \ell ' +1 )}. $$ Then, since $a \in \scr{D} _\scr{A} (A) \cap \scr{D} _\scr{A} (A')$, it follows that the upper left block,
$$ \bpm A (a) & 0 \\ 0 & A' (a) \epm \in \scr{A} ^{(\ell + \ell') \times (\ell + \ell ')}, $$ is invertible over $\scr{A}$. Moreover, the Schur complement of $M (a)$ by this block is 
\ba &  & D' (a) - D(a) - (C (a), C' (a)) \bpm A(a) ^{-1} & 0 \\ 0 & A' (a) ^{-1} \epm \bpm B(a) \\ -B' (a) \epm \\
& = & D(a) - C(a) A(a) ^{-1} B (a) - D' (a) + C' (a) A' (a) ^{-1} B' (a) \neq 0, \ea  
so that by \cite[Proposition 5.4.6]{Cohn}, the inner rank of $M(a)$ over $\scr{A}$ is $\rho _\scr{A} (M(a)) = \ell + \ell ' +1$. However, $\scr{M} _d (r)$ is a skew field, hence stably-finite, so that as a matrix over $\scr{M} _d (r)$, $M (\fz)$ is still such that the upper left block, $\bsm A (\fz) & 0 \\ 0 & A' (\fz) \esm$ is invertible, but now the Schur complement by this upper left block is 
$$  D(\fz) - C(\fz) A(\fz) ^{-1} B (\fz) - \left( D' (\fz) - C' (\fz) A' (\fz) ^{-1} B' (\fz) \right) = f(\fz) - f (\fz) = 0,$$ and 
\cite[Proposition 5.4.6]{Cohn} now implies that the inner rank of $M(\fz)$ over $\scr{M} _d (r)$ is $\ell + \ell ' < \ell + \ell ' +1$. Since $\scr{O} _d (r)$ is a semifir and $\scr{M} _d (r)$ is its universal skew field of fractions, the embedding $\scr{O} _d (r) \hookrightarrow \scr{M} _d (r)$ is honest and hence inner-rank preserving by \cite[Proposition 5.4.10]{Cohn}, and we conclude that $M(\fz)$ has inner rank equal to $\ell + \ell'$ over $\scr{O} _d (r)$. Hence, $M(\fz)$ has the inner-rank factorization $M(\fz) = M'(\fz) M'' (\fz)$ over $\scr{O} _d (r)$, with $M'(\fz) \in \scr{O} _d (r) ^{(\ell + \ell' +1) \times (\ell + \ell ')}$ and $M'' (\fz) \in \scr{O} _d (r) ^{(\ell + \ell ') \times (\ell + \ell ' +1)}$. Evaluation at $a \in \scr{A} ^{1\times d}$ (which is valid for matrices over $\scr{O} _d (r)$, by assumption) then yields $M (a) = M'(a) M''(a)$, so that $\rho _\scr{A} (M(a)) = \ell + \ell'$ contradicting that $\rho _\scr{A} (M(a)) = \ell + \ell ' +1$. 
\end{proof}

\begin{remark}
The topological $*-$algebra, $\wt{\mf{M}}$, of closed linear operators affiliated to the finite von Neumann algebra, $\mf{M}$, is stably-finite by \cite[Lemma 5.3]{freefield}, and so the previous theorem applies to potential evaluations of elements of $\scr{M} _d (r)$ at tuples $\Xi \in \wt{\mf{M}} ^{1\times d}$, provided that this evaluation extends to a homomorphism on $\scr{O} _d (r)$. Certainly if $\Xi \in \mf{M} ^{1 \times d}$ obeys $\| \Xi \| _{\mr{row}} < r$, then the homomorphism $\mr{ev} _\Xi : \scr{O} _d (r) \rightarrow \mf{M}$ is well-defined by \cite[Theorem 1.1]{Pop-freeholo}, since any $f \in \scr{O} _d (r)$ has radius of convergence at least $r$. 
\end{remark}

The following is an immediate corollary of Theorem \ref{main1} and Theorem \ref{evalthm}, and it generalizes (v) $\Rightarrow$ (i) of \cite[Theorem 1.1]{freefield}.

\begin{cor}\label{cor:embedding_skew_field}
Let $\Xi \in \mf{M} ^{1\times d}$ with $\|\Xi\|_\mathrm{row}<r$ and $\Delta_1 (\Xi) =d$. Then the homomorphism $\mathrm{ev}_\Xi:\scr{O} _d (r) \rightarrow \mf{M}$ given by evaluation at $\Xi$ extends to an embedding $\mathrm{Ev}_\Xi:\scr{M} _d (r) \hookrightarrow \wt{\mf{M}}$. 
\end{cor}

The following corollary is an extension of \cite[Theorem 3]{CharShly-analytic} and \cite[Theorem 3.1]{MSW-alg_rel}. We say that an element $f \in \cO_d(r)$ is self-adjoint, if it is invariant under the natural anti-linear involution on power series that keeps the variables invariant. In particular, if $f \in \scr{O} _d (r)$ is self-adjoint and $\Xi \in \mf{M} ^{1\times d}$ is self-adjoint, then $f (\Xi) = f(\Xi) ^* \in \mf{M}$ is self-adjoint.

\begin{cor} \label{cor:max_ent_dim}
    Let $\Xi \in \mf{M}^{1\times d}$ be a $d$-tuple of self-adjoint elements, such that $\|\Xi\|_{\mathrm{row}} < r$ and $\delta^*(\Xi) = d$. Let $f \in \cO_d(r)$ be self-adjoint. Then, for every $p \in \mf{M}$, if $f(\Xi) p = 0$, then $f = 0$. In particular, for every $0 \neq f \in \cO_d(r)$ self-adjoint, $\delta^*(f(\Xi)) = 1$ and $f(\Xi)$ is diffuse.
\end{cor}

\begin{proof}
    The result follows immediately from the preceding theorem once we observe that by \cite{CharShly-analytic}, $\delta^*(\Xi) \leq \Delta_1(\Xi)$.
\end{proof}

Following Mai, Speicher and Yin, we introduce the analytic rank of matrices over $\wt{\mf{M}}$.

\begin{defn}\label{def:analytic_rank}
    Let $H\in \wt{\mf{M}}^{\ell\times \ell}$. The \emph{analytic rank} of $H$ is
    \begin{equation*}
        \mathrm{rank}(H)=\mathrm{Tr}_\ell\circ \tau^{(\ell)}(P_{\overline{\mathrm{Im}(H)}}).
    \end{equation*}
    Here, $\mathrm{Tr}_\ell$ denotes the non-normalized trace on $\C^{\ell\times \ell}$, $\tau^{(\ell)}$ is an ampliation of the faithful, normal tracial state $\tau$, and $P_{\overline{\mathrm{Im}(H)}}$ denotes the projection onto the closure of $\mathrm{Im}(H)$.  
\end{defn}

\begin{remark}
    The fact that the  projection $P_{\overline{\mathrm{Im}(H)}}$ from Definition \ref{def:analytic_rank} is an element of $\mf{M}^{\ell\times \ell}$ was proved in \cite[Lemma 5.2]{freefield}. Moreover, \cite[Lemma 5.2]{freefield} shows that the projection onto $\mathrm{ker}(H)$ lies in $\mf{M}^{\ell\times \ell}$, and 
    \begin{equation*}
        \mathrm{rank}(H)=\ell-\mathrm{Tr}_\ell\circ \tau^{(\ell)}(P_{\mathrm{ker}(H)}).
    \end{equation*}
\end{remark}

This notion of analytic rank is useful since it yields a condition for invertibility of matrices over $\wt{\mf{M}}$.

\begin{lem}[{\cite[Lemma 5.5]{freefield}}]\label{lem:analytic_rank_invertibility}
    $H\in \wt{\mf{M}}^{\ell\times \ell}$ is invertible if and only if $\mathrm{rank}(H)=\ell$.
\end{lem}

The following is an analogue of \cite[Theorem 5.6]{freefield} in our setting.

\begin{thm}
    Let $\Xi \in \mf{M} ^{1\times d}$ with $\|\Xi\|_\mathrm{row}<r$. The following are equivalent.
    \bi
    \item[(i)] For any $\ell\in\N$ and $H\in \scr{O} _d (r) ^{\ell \times \ell}$ we have: if $H$ is full, then $H(\Xi)\in \wt{\mf{M}}^{\ell\times \ell}$ is invertible.
    \item[(ii)] For any $\ell\in \N$ and $H\in \scr{O} _d (r) ^{\ell \times \ell}$, $\mathrm{rank} \, H(\Xi) =\rho(H)$, where $\rho$ denotes the inner rank over $\scr{O} _d (r)$.
    \item[(iii)] The homomorphism $\mathrm{ev}_\Xi:\scr{O} _d (r) \rightarrow \wt{\mf{M}}$ given by evaluation at $\Xi$ extends to an embedding $\mathrm{Ev}_\Xi:\scr{M} _d (r) \hookrightarrow \wt{\mf{M}}$.
    \ei
    Moreover, if the equivalent conditions above hold, then for all $H\in \scr{O} _d (r) ^{\ell \times \ell}$,
    \begin{equation*}
        \mathrm{rank}\, H(\Xi)=\rho(H)=\rho_{\wt{\mf{M}}}(H(\Xi)),
    \end{equation*}
    where $\rho_{\wt{\mf{M}}}$ denotes the inner rank over $\wt{\mf{M}}$.
\end{thm}

\begin{proof}
    The implication (ii)$\implies$(i) follows from Lemma \ref{lem:analytic_rank_invertibility}, while the implication (i)$\implies$(iii) follows as in the proof of Corollary \ref{cor:embedding_skew_field}. To show (iii)$\implies$(ii) one can simply follow the proof of (iv)$\implies$(iii) in \cite[Theorem 5.6]{freefield}. This works since the facts used in the argument are general facts that hold for semifirs and their universal skew fields. 

    In the case when the equivalent conditions (i), (ii) and (iii) are satisfied, one can deduce the equality of the different ranks by again following the proof of \cite[Theorem 5.6]{freefield}.
 \end{proof}

\begin{cor} \label{cor:atom_weight}
    Let $\Xi \in \mf{M}^{1\times d}$ with $\|\Xi\|_{\mathrm{row}} < r$ be a tuple of self-adjoints. Let $H \in \cO_d(r)^{\ell \times \ell}$ be self-adjoint. Let $\mu$ be the spectral measure of $H(\Xi)$. Then, for every $t \in \R$, $\mu(\{t\}) \in \{0,\frac{1}{\ell},\cdots,\frac{\ell - 1}{\ell}, 1\}$.
\end{cor}
\begin{proof}
    By replacing $H$ with $H-t$ it suffices to consider $\mu(\{0\})$. Now 
    \[
    \mu(\{0\}) = \tr_{\ell} \circ \tau^{(\ell)}(P_{\ker(H)}) = \frac{\ell - \rho(H)}{\ell} = 1 - \frac{\rho(H)}{\ell} \in \{0,\frac{1}{\ell},\cdots,\frac{\ell - 1}{\ell}, 1\}.
    \]
\end{proof}

\section{The inequality $\Delta_1 < \Delta$} \label{sec:inequality}

Let $f(x,y) = \sum_{n=0}^{\infty} \frac{1}{n!} x^n y^n \in \cO_2$.  We will use the fact that $f$ resembles an exponential. As is well known, the exponential is a transcendental function and therefore, the graph of the exponential is Zariski-dense. Our goal is to create operators that do not satisfy any finite-rank relations, but satisfy a trace-class one.

\begin{lem} \label{lem:random_basis}
For every $m \in \N$ and a generic $T \in \GL_m(\C)$, the matrices $\{e_i e_i^* T e_j e_j^* T^{-1}\}_{i,j=1}^m$ are linearly independent and, thus, form a basis for $\C ^{m\times m}$.    
\end{lem}
\begin{proof}
    Write $T = (t_{ij})_{i,j=1}^m$. Write $T = (t_{ij})$ and $T^{-1} = \frac{1}{\det T} (\tau_{ij})$, where each $\tau_{ij}$ is a polynomial in $t_{ij}$. Write
    \[
    S_{ij} = e_i e_i^* T e_j e_j^* T^{-1} = t_{ij} e_i e_j^* T^{-1}.
    \]
    Note that right multiplication by an invertible matrix is an invertible linear map on $\C ^{m \times m}$. Hence, since $S_{ij} T = t_{ij} e_i e_j^{*}$ the set $\{ S _{i,j} \}$ is linearly independent if and only if all $t_{ij}$ are non-zero.
\end{proof}

\begin{lem} \label{lem:alg_indep}
    Let $p_1,\cdots,p_n \in \C[z_1,\cdots,z_d]$ be a set of polynomials, such that the set $\{1, p_1, \cdots, p_n\}$ is linearly independent. Then, the set of functions $e^{p_1}, \cdots, e^{p_n}$ is algebraically independent and transcendental over $\C(z_1,\cdots,z_d)$.
\end{lem}
\begin{proof}
    Assume that our functions, $e^{p_i}$, satisfy a polynomial relation $q 
    \in \C(z_1,\cdots,z_d)[w_1,\cdots, w_n]$ (the coefficients are rational functions of the $z_i$). Clearing denominators, we may assume that the coefficients of $q$ are polynomials. Write $q = \sum_{\alpha} q_{\alpha} w^{\alpha}$. Hence,
    \[
    0 = q(e^{p_1},\cdots,e^{p_n}) = \sum_{\alpha} q_{\alpha} e^{\sum_{j=1}^n \alpha_j p_j}.
    \]
    In other words, the functions $e^{\sum_{j=1}^n \alpha_j p_j}$, where we run over the multi-indices of $q$ with non-zero coefficients are linearly dependent over $\C[z_1,\cdots,z_d]$. By a theorem of Narasimhan \cite[Theoreme 1]{Narasimhan-alg_indep}, this can only happen, if there exist two distinct indices $\alpha$ and $\beta$, such that $\sum_{j=1}^n (\alpha_j - \beta_j) p_j$ is a scalar. This contradicts our assumption of linear independence and the proof is complete.
\end{proof}

\begin{lem} \label{lem:determinant_hypersurface}
    For any $n \in \N$, if $P \in \C\langle x,y,z \rangle ^{n \times n}$ is a full matrix, then there exists $m \in \N$, and $(X,Y) \in \C ^{m \times m} \otimes \C ^{1\times 2}$, such that $\det P(X,Y,f(X,Y)) \neq 0$. Hence, $P(x,y,f(x,y)) \in \cO_2 ^{n\times n}$ is full.
\end{lem}
\begin{proof}
    To prove the claim it suffice to show that for every $m \in \N$, the graph of $f$ is Zariski dense in $\C ^{(m\times m)\cdot 3}/\GL_m(\C)$. Namely, if 
    \[
    \Gamma_m(f) = \left\{ (X, Y, f(X,Y)) \mid X, Y \in \C ^{m\times m} \right\},
    \]
    then we want to show that $\overline{\pi(\Gamma_m(f))}^{\mr{Zar}} = \C ^{(m\times m)\cdot 3}/\GL_m(\C)$, where $\pi$ is the map dual to the embedding of $\GL_m(\C)$-invariant polynomials in the polynomial ring on the coordinates of the matrices in $\C ^{(m \times m) \cdot 3}$. Indeed, if we know this claim, then given a full $P \in\C\langle x,y,z\rangle ^{n \times n}$, we know that there exists some $m_0$, such that $\det P(\Omega_x^{(m_0)}, \Omega_y^{(m_0)}, \Omega_z^{(m_0)}) \neq 0$, where the $\Omega$'s are the generic $m_0 \times m_0$ matrices. Since this is a non-trivial $\GL_m(\C)$-invariant polynomial, it cannot vanish identically on $\Gamma_{m_0}(f)$. This proves the claim.

    To prove the claim about the graph, we fix $m \in \N$ and note that the set of all points, such that the first two matrices are simple (namely, diagonalizable with distinct eigenvalues) is dense in $\C ^{(m\times m)\cdot 3}$ in the classical topology and, thus, also in the Zariski topology. Applying a similarity we can change every such point to another point, where both $X$ and $Y$ are simple and $X$ is diagonal. Write $\lambda_1,\cdots,\lambda_m$ and $\mu_1,\cdots,\mu_m$ for the eigenvalues of $X$ and $Y$ respectively. In particular, there is an invertible matrix $T$, such that
    \[
    X = \sum_{j=1}^m \lambda_j e_j e_j^* \text{ and } Y = \sum_{j=1}^m \mu_j T e_j e_j^* T^{-1}.
    \]
    Of course, $T$ is not uniquely defined and we may assume that $\det(T) = 1$. Therefore, $T^{-1}$ is the adjugate matrix of $T$ and its entries are polynomials in the entries of $T$. Moreover, the entries of $T$ are independent of the variables of the eigenvalues of $X$ and $Y$. In this presentation for every $n \in \N$,
    \[
    X^n Y^n = \left(\sum_{i=1}^m \lambda_i^n e_i e_i^* \right) \left( \sum_{j=1}^m \mu_j^n T e_j e_j^* T^{-1} \right) = \sum_{i,j = 1}^m \lambda_i^n \mu_j^n e_i e_i^* T e_j e_j^* T^{-1}.
    \]
    Therefore,
    \[
    f(X,Y) = \sum_{n=0}^{\infty} \frac{1}{n!} X^n Y^n = \sum_{i,j =1}^m e^{\lambda_i \mu_j} e_i e_i^* T e_j e_j^* T^{-1}.
    \]
    let $p$ be a $\GL_n(\C)-$invariant polynomial that vanishes on $\Gamma_m(f)$. By the results of Procesi, the ring of $\GL_n(\C)-$invariant polynomials is generated by traces of monomials in generic matrices \cite[Theorem 1.3]{Procesi-invariant_theory}. In particular, if we evaluate $p$ on $(X,Y,f(X,Y))$, where $X$ and $Y$ are as above, then $p$ gives us an algebraic relation in $\lambda_i$, $\mu_j$, the coordinates of $T$ and $e^{\lambda_i \mu_j}$, the latter are, however, transcendental and algebraically independent over the rational functions in the variables by Lemma \ref{lem:alg_indep}. Therefore, if we replace the $e^{\lambda_i \mu_j}$ by independent variables $z_{ij}$, we obtain that
    \[
    p\left( X,Y, \sum_{i,j= 1}^m z_{ij} e_i e_i^* T e_j e_j^* T^{-1} \right) = 0.
    \]
    By the preceding lemma, for a generic $T$, the matrices $e_i e_i^* T e_j e_j^* T^{-1}$ for a basis for $\C ^{m\times m}$. Varying the $z_{ij}$ makes the last coordinate arbitrary on a Zariski dense set of $Y$'s. Hence, this polynomial vanishes on a Zariski dense set and, therefore, $p=0$.
    \end{proof}

\begin{thm} \label{thm:counter_example}
    Let $\Xi_1,\Xi_2$ be standard free semicirculars and let $\Xi_3 = f(\Xi_1,\Xi_2)$, where $f(x,y) = \sum_{n=0}^{\infty} \frac{1}{n!} x^n y^n \in \cO_2$. Then, $$\Delta_1(\Xi_1,\Xi_2,\Xi_3) < 3 = \Delta(\Xi_1,\Xi_2,\Xi_3).$$
\end{thm}
\begin{proof}
    The preceding lemma combined with the fact that every full matrix over $\cO_2$ evaluated on $\Xi_1,\Xi_2$ is invertible as a closed operator affiliated to the von Neumann algebra generated by $\Xi_1$ and $\Xi_2$. Therefore, we can apply \cite[Theorem 1.1]{freefield} to conclude that $\Delta(\Xi_1,\Xi_2,\Xi_3) = 3$. However, since there is an analytic relation between the $\Xi_i$, \cite[Theorem 4]{CharShly-analytic} or Theorem \ref{main1} imply that $\Delta_1(\Xi_1,\Xi_2,\Xi_3) < 3$.
\end{proof}

\begin{remark}
    The same argument can be applied to the function $g(x,y) = \sum_{n=0}^{\infty} \frac{1}{n!} x^n y^n x^n$, since for $X$ and $Y$ both simple (as in the proof above), we have that 
    \[
    X^n Y^n X^n = \sum_{i,j,k=1}^m \lambda_i^n \mu_j^n \lambda_k^n e_i e_i^* T e_j e_j^* T^{-1} e_k e_k^* = \sum_{i,j,k=1}^m \lambda_i^n \mu_j^n  \lambda_k^n t_{ij} \tilde{t}_{jk} e_i e_k^*
    \]
    Here, $\tilde{t}_{jk}$ is the $j,k$-the entry of $T^{-1}$. Hence,
    \[
    f(X,Y) = \sum_{i,j,k=1}^m e^{\lambda_i \mu_j \lambda_k} t_{ij} \tilde{t}_{jk} e_i e_k^*.
    \]
    Now the matrices $\{t_{ij} \tilde{t}_{jk} e_i e_k^*\}_{i,j,k=1}^m$ span the $m\times m$ matrices for generic choice of $T$.

    The advantage of this function is that for self-adjoint $X$ and $Y$, $f(X,Y)$ is also self-adjoint. Therefore, if $\Xi_1, \Xi_2$ are free semicirculars, then 
    \[
    \delta^*(\Xi_1,\Xi_2,g(\Xi_1,\Xi_2)) \leq \Delta_1(\Xi_1,\Xi_2,g(\Xi_1,\Xi_2)) < 3 = \Delta (\Xi_1,\Xi_2,g(\Xi_1,\Xi_2)).
    \]
    Moreover, by \cite[Proposition 6.4]{Voic-entropy_II} and \cite[Corollary 6.13]{Voic-entropy_II}, 
    \[
    \delta^*(\Xi_1,\Xi_2,g(\Xi_1,\Xi_2)) = \delta^*(\Xi_1,\Xi_2) = 2.
    \]
\end{remark}

\section{Tool and computational resource disclosure}

In compliance with the Leiden declaration on AI and Mathematics, we include this section. We have used ChatGPT v5.6 Sol for literature searches and for ideas on proving Lemma \ref{lem:determinant_hypersurface}. An earlier version of Theorem 7.4 for free semicirculars was obtained by ChatGPT v5.6 Sol. However, as we later realized the construction used was already in the work of K. Dykema \cite{Dykema-free_prod}. The paper and the proofs were written by the authors and we take full responsibility for them.

\setstretch{1}
\setlength{\parskip}{0pt}
\setlength{\itemsep}{0pt}


\end{document}